\documentclass{amsart}

\usepackage{amssymb,amsmath}

\newcommand{\Tm}{\sqrt{-\Delta + m^2} \, }
\newcommand{\Tmc}{\sqrt{-c^2 \Delta + m^2 c^4} \, }
\newcommand{\Tmcn}{\sqrt{-c_n^2 \Delta + m^2 c_n^4} \, }
\newcommand{\RR}{\mathbf{R}}
\newcommand{\CC}{\mathbf{C}}
\newcommand{\Hhalf}{H^{1/2}(\RR^3)}
\newcommand{\Honer}{H^{1}_{\mathrm{r}}(\RR^3)}
\renewcommand{\to}{\rightarrow}
\newcommand{\weakto}{\rightharpoonup}

\newcommand{\Res}{\mathcal{R}}

\newcommand{\En}{\mathcal{E}}
\newcommand{\Enc}{\mathcal{E}_{c}}
\newcommand{\Encn}{\mathcal{E}_{c_n}}
\newcommand{\Enr}{\mathcal{E}_{\mathrm{nr}}}

\newcommand{\Hell}{\mathcal{H}_{(\ell)}}

\newcommand{\Ncrit}{N_*}
\newcommand{\Qc}{Q_c}
\newcommand{\Qcn}{Q_{c_n}}
\newcommand{\Qinf}{Q_\infty}

\newcommand{\ie}{i.\,e.}

\newtheorem{theorem}{Theorem}
\newtheorem{lemma}{Lemma}
\newtheorem{proposition}{Proposition}

\newtheorem*{remark*}{Remark}
\newtheorem*{remarks*}{Remarks}

\numberwithin{equation}{section}

\begin{document}

\title[Uniqueness of Ground States for Pseudo-Rel.~Hartree Equations]{Uniqueness of Ground States for Pseudo-Relativistic Hartree Equations}

\author{Enno Lenzmann}
\thanks{Partly supported by NSF Grant DMS--0702492.}
\date{September 15, 2008}
\email{lenzmann@math.mit.edu}
\address{Massachusetts Institute of Technology, Department of Mathematics, Room 2-230, Cambridge, MA 02139, USA.}

\maketitle

\begin{abstract}
We prove uniqueness of ground states $Q \in H^{1/2}(\RR^3)$ for the pseudo-relativistic Hartree equation,
\[
\sqrt{-\Delta + m^2} \, Q - \big ( |x|^{-1} \ast |Q|^2 \big ) Q = -\mu Q ,
\]
in the regime of $Q$ with sufficiently small $L^2$-mass. This result shows that a uniqueness conjecture by Lieb and Yau in [CMP \textbf{112} (1987), 147--174] holds true at least for $N=\int |Q|^2 \ll 1$ except for at most countably many $N$.

Our proof combines variational arguments with a nonrelativistic limit, which leads to a certain Hartree-type equation (also known as the Choquard-Pekard or Schr\"odinger-Newton equation). Uniqueness of ground states for this limiting Hartree equation is well-known. Here, as a key ingredient, we prove the so-called nondegeneracy of its linearization. This nondegeneracy result is also of independent interest, for it proves a key spectral assumption in a series of papers on effective solitary wave motion and classical limits for nonrelativistic Hartree equations. 
\end{abstract}

\section{Introduction}

The pseudo-relativistic Hartree energy functional (in appropriate units)
\begin{equation} \label{eq:En}
\En(\psi) = \int_{\RR^3} \overline{\psi} \Tm \psi - \frac{1}{2} \int_{\RR^3} \big (|x|^{-1} \ast |\psi|^2 \big ) |\psi|^2 
\end{equation}
arises in the mean-field limit of a quantum system describing many self-gravitating, relativistic bosons with rest mass $m >0$. Such a physical system is often referred to as a {\em boson star,} and various models for these -- at least theoretical -- objects have attracted a great deal of attention in theoretical and numerical astrophysics over the past years.  

In order to gain some rigorous insight into the theory of boson stars, it is of particular interest to study ground states (i.\,e.~minimizers) for the variational problem 
\begin{equation} \label{eq:var}
E(N) = \inf \Big \{ \En(\psi ) : \mbox{$\psi \in H^{1/2}(\RR^3)$ and $\displaystyle \int_{\RR^3} |\psi|^2 = N$} \Big \} ,
\end{equation}
where the parameter $N >0$ plays the role of the stellar mass. Provided that problem (\ref{eq:var}) has indeed a ground state $Q \in H^{1/2}(\RR^3)$, one readily finds that it satisfies the {\em pseudo-relativistic Hartree equation,}
\begin{equation} \label{eq:bs}
\Tm Q - \big (|x|^{-1} \ast |Q|^2 ) Q = -Ê\mu Q,
\end{equation} 
with $\mu = \mu(Q) \in \RR$ being some Lagrange multiplier. 

\medskip
In fact, the existence of symmetric-decreasing ground states $Q=Q^*(|x|) \geq 0$ minimizing (\ref{eq:var}) was first proven by Lieb and Yau in \cite{LiebYau1987}, where the authors also conjectured that uniqueness result holds true in the following sense. For each $N >0$, the variational problem (\ref{eq:var}) has at most one symmetric-decreasing ground state. If true, this result further implies, by strict rearrangement inequalities, that we have indeed uniqueness of all the ground states of (\ref{eq:var}) for each $N >0$, up to phase and translation. 
 
\medskip
However, the nonlocality of $\sqrt{-\Delta + m^2}$ as well as the convolution-type nonlinearity both complicate the analysis of the pseudo-relativistic Hartree equation (\ref{eq:bs}) in a substantial way. In particular,  the set of its radial solutions is not amenable to ODE techniques (e.\,g., shooting arguments and comparison principles) which are key arguments for proving uniqueness of ground states for nonlinear Schr\"odinger equations (NLS) with local nonlinearities; see \cite{Kwong1989,McLeod1993,McLeodSerrin1987,PeletierSerrin1983}. 

A further complication in the analysis of equation (\ref{eq:bs}) stems from the fact that there are no simple scaling arguments that relate ground states with different $N$, due to the presence of $m >0$. Indeed, this lack of a simple scaling mechanism is essential for the existence of a critical stellar mass $\Ncrit > 0$; see Theorem \ref{th:intro} below.

\medskip
As a first step towards proving uniqueness of ground states for (\ref{eq:var}), we present Theorem \ref{th:unique} below, which shows that ground states for problem are indeed unique (modulo translation and phase) for all sufficiently small $N > 0$ except for at most countably many. Our proof uses variational arguments combined with a nonrelativistic limit, leading to the nonlinear Hartree equation (also called Choquard-Pekar or Schr\"odinger-Newton equation) given by
\begin{equation} \label{eq:liebintro}
-\frac{1}{2m} \Delta \Qinf - \big ( |x|^{-1} \ast |\Qinf|^2 \big ) \Qinf = - \lambda \Qinf .
\end{equation}  
It is known this equation has a unique radial, positive solution $\Qinf \in H^1(\RR^3)$ for $\lambda > 0$ given; see \cite{Lieb1977} and Appendix A. 

In the present paper, we prove (as a key ingredient) that $\Qinf \in H^1(\RR^3)$ has a {\em nondegenerate linearization.} By this we mean that the linearization of (\ref{eq:liebintro}) around $\Qinf$ has a nullspace that is entirely due to the equation's invariance under phase and translation transformation; see Theorem \ref{th:nondegnr} below and its remarks for a precise statement. In particular, we show that the linear operator $L_+$ given by
\begin{equation} \label{eq:LLL}
L_+ \xi = -\frac{1}{2m} \Delta \xi + \lambda \xi - \big (|x|^{-1} \ast |\Qinf|^2 \big ) \xi - 2 \Qinf \big ( |x|^{-1} \ast (\Qinf \xi) \big )
\end{equation}
satisfies
\begin{equation} \label{eq:nondeg_LL}
\mathrm{ker} \, L_+ = \mathrm{span} \, \{ \partial_{x_1} \Qinf, \partial_{x_2} \Qinf, \partial_{x_3} \Qinf \} .
\end{equation}
Furthermore, by a perturbation argument, we conclude an analogous nondegeneracy result for ground states of the pseudo-relativistic Hartree equation (\ref{eq:bs}) with sufficiently small $L^2$-mass; see Theorem \ref{th:nondeg} below.

\medskip
In addition to being a mere technical key fact proven in this paper, the nondegeneracy result for equation (\ref{eq:liebintro}) is also of independent interest. For example, it proves a key spectral assumption in a series of papers on effective solitary wave motion and classical limits for Hartree equations; see \cite{FTY2002,JFGS2004,JFGS2006,Walid2007} and also the remark following Theorem \ref{th:nondegnr}. Another very recent application of the nondegeneracy result (\ref{eq:nondeg_LL}) is presented in \cite{KriegerMartelRaphael2008}, where two soliton solutions to the time-dependent version of equation (\ref{eq:liebintro}) are constructed.

\medskip
In the context of ground states for NLS with {\em local nonlinearities,} the nondegeneracy of linearizations is a well-known fact (see \cite{CGNT2007,Weinstein1985}) and it plays a central role in the stability analysis of solitary waves for NLS. However, the arguments for NLS with local nonlinearities make use of Sturm-Liouville theory, which, by contrast, is not applicable to $L_+$ given by (\ref{eq:LLL}) due to its nonlocal character. For more details, we refer to Section \ref{sec:proof:th:nondeg} below.

\medskip
Apart from their minimizing property, the ground states for (\ref{eq:var}) also play an important role for the {\em time-dependent pseudo-relativistic Hartree equation,}
\begin{equation} \label{eq:thartree}
i \partial_t \psi = \Tm \psi - \big ( |x|^{-1} \ast |\psi|^2 \big ) \psi,
\end{equation}
with the wave field $\psi : [0,T) \times \RR^3 \to \CC$. Clearly, equation (\ref{eq:thartree}) has solitary wave solutions
\begin{equation}
\psi(t,x) = e^{it \mu} Q(x) ,
\end{equation} 
whenever $Q \in H^{1/2}(\RR^3)$ is a nontrivial solution to equation (\ref{eq:bs}). Let us also mention that the dispersive nonlinear PDE (\ref{eq:thartree}) exhibits a rich variety of phenomena, such as stable and unstable traveling solitary waves, as well as finite-time blowup solutions indicating the ``gravitational collapse'' of a boson star; see \cite{FJLCMP2007, FJLNonlinearity2007, FroehlichLenzmannCPAM2007}. For well-posedness results concerning equation (\ref{eq:thartree}) and its rigorous derivation from many-body quantum mechanics, we refer to \cite{ChoOzawa2006,LenzmannWP2007} and \cite{ElgartSchlein2007}, respectively.

\medskip
For the reader's convenience, we conclude our introduction by summarizing the existence result about ground states for problem (\ref{eq:var}) along with a list of their basic properties. 

\begin{theorem}{\bf (Existence and Properties of Ground States.)}   \label{th:intro}
Suppose that $m > 0$ holds in (\ref{eq:En}). Then there exists a universal constant $\Ncrit > 4/\pi$ (independent of $m$) such that the following holds. 

\begin{enumerate}
\item[(i)] {\bf Existence:} There exists a ground state $Q \in H^{1/2}(\RR^3)$ for problem (\ref{eq:var}) if and only if 
\[
0 < N < \Ncrit .
\]
Moreover, the function $Q$ satisfies the pseudo-relativistic Hartree equation (\ref{eq:bs}) in the sense of distributions with some Lagrange multiplier $\mu \in \RR$.

\item[(ii)] {\bf Smoothness and Exponential Decay:} Any ground state $Q$ belongs to $H^s(\RR^d)$ for all $s \geq 0$ and $e^{+\delta |x|} Q \in L^\infty(\RR^3)$ for some $\delta = \delta(Q) > 0$.

\item[(iii)] {\bf Radiality and Strict Positivity:} Any ground state $Q$ is equal to its spherical-symmetric rearrangement $Q^*(|x|)$ up to phase and translation. Moreover, we have $Q^*(|x|) > 0$ for all $x \in \RR^3$.
\end{enumerate}
\end{theorem}

\begin{remark*}
{\em  For the proof of (i) and (ii)/(iii), we refer to \cite{LiebYau1987} and \cite{FJLCMP2007,LenzmannDiss2006}, respectively. In physical terms, the constant $\Ncrit > 0$ can be regarded as the {\em ``Chandrasekhar limit mass''} of a pseudo-relativistic boson star.

}
\end{remark*}

\section{Main Results}

We now state our first main result concerning the uniqueness of ground states for the pseudo-relativistic Hartree equation (\ref{eq:bs}).

\begin{theorem} {\bf (Uniqueness of Ground States for $N \ll 1$.)} \label{th:unique}
Assume that $m > 0$ holds in (\ref{eq:En}). Then, for $0 < N \ll 1$, we have uniqueness of ground states for problem (\ref{eq:var}) up to phase and translation whenever $E'(N)$ exists. In particular, the symmetric-decreasing ground state $Q=Q^* \in H^{1/2}(\RR^3)$ minimizing (\ref{eq:var}) is unique for such $N > 0$.
\end{theorem}

\begin{remarks*}
{\em 1) Since it known from \cite{LiebYau1987} that the ground state energy $E(N)$ is strictly concave, the derivative $E'(N)$ exists for all $N \in (0,N_*)$, except on a subset $\Sigma$ which is at most countable. In particular, it is easy to see that the Lagrange multiplier $\mu$ is unique for such $N \in (0,N_*) \setminus \Sigma$, in the sense that $\mu$ only depends on $Q$ through $N = \int |Q|^2$. Our argument to prove Theorem \ref{th:unique} has to avoid the ``exceptional'' set $\Sigma$. A natural conjecture would be that  $\Sigma = \emptyset$ holds.

2) It would be desirable to extend this uniqueness result (whose proof partly relies on perturbative arguments) to the whole range $0 < N < \Ncrit$ of existence; or, more interestingly, to disprove uniqueness for some $N >0$ sufficiently large. 

3) Note that, by definition, ground states for the pseudo-relativistic Hartree equation (\ref{eq:var}) are always minimizers for the variational problem (\ref{eq:var}). In principle, we cannot exclude the possibility that equation (\ref{eq:bs}) has a positive solution without being a minimizer for (\ref{eq:var}).

4) To the author's knowledge, this is the first uniqueness result for ground states that solve a nonlinear pseudo-differential equation in space dimension $n > 1$. In fact, apart from a very special case arising in $n=1$ dimensions for solitary waves solving Benjamin-Ono-type equations (see \cite{AmickToland1991, Albert1995}), nothing seems to be known, for instance, about uniqueness of ground states $\varphi \in H^s(\RR^n)$ for nonlinear equations involving the fractional Laplacian:
\begin{equation} \label{eq:pseudoground}
(-\Delta)^{s/2} \varphi + f(\varphi) = -\mu \varphi,
\end{equation} 
where $f(\varphi)$ denotes some nonlinearity and $\mu \in \RR$ is given. The author plans to pursue this question in future work.

4) We remark that if $m=0$ vanishes, we have existence of ground states for problem (\ref{eq:var}) if and only if $N= \Ncrit$ holds. In what follows, we shall exclusively deal with the physically relevant case where $m >0$ holds. Nevertheless, it remains an interesting open question whether uniqueness of ground states also holds for $m=0$, since the methods developed here are clearly not applicable to this limiting case.
}
\end{remarks*}

Our next result proves a so-called nondegeneracy condition, which was introduced in \cite{FJLNonlinearity2007} as a spectral assumption supported by numerical evidence. There, the effective motion of solitary waves for equation (\ref{eq:thartree}) with an slowly varying external potential was studied. Furthermore, the following nondegeneracy result allows us to give an unconditional proof for the {\em cylindrical symmetry of traveling solitary waves} for the time-depenendent pseudo-relativistic Hartree equation (\ref{eq:thartree}); see \cite{FJLNonlinearity2007} for more details. The precise nondegeneracy statement reads as follows.

\begin{theorem} {\bf (Nondegeneracy of Ground States for $N \ll 1$.)} \label{th:nondeg}
Let $m > 0$ in (\ref{eq:En}) and suppose that $Q = Q^*$ is a symmetric-decreasing ground state for problem (\ref{eq:var}) with Lagrange multiplier $\mu \in \RR$. Furthermore, we consider the linear operator $L_+$ given by
\[
L_+ \xi = \big ( \sqrt{-\Delta + m^2} + \mu \big ) \xi - \big ( |x|^{-1} \ast |Q|^2 \big ) \xi - 2 Q \big ( |x|^{-1} \ast (Q \xi) \big ),
\]
acting on $L^2(\RR^3)$ with domain $H^1(\RR^3)$. Then, for $0 < N \ll 1$, the operator $L_+$ is nondegenerate, i.\,e., its kernel satisfies
\[
\mathrm{ker} \, L_+ = \mathrm{span} \, \big \{ \partial_{x_1} Q, \partial_{x_2} Q, \partial_{x_3} Q \big \}. 
\]
\end{theorem}

\begin{remarks*}
{\em 1) This completely characterizes the kernel of the linearization of the pseudo-relativistic Hartree equation (\ref{eq:bs}) around ground state $Q=Q^*$ with $\int |Q|^2 \ll 1$. Note that, due to the presence of $|Q|^2$ in the nonlinearity, the linearized operator is not $\CC$-linear. See also the remark following Theorem \ref{th:nondegnr} below for more details on the analogous statement for the nonrelativistic equation (\ref{eq:liebintro}).

2) Note that the nondegeneracy of $L_+$ holds for all $N= \int |Q|^2 \ll 1$. The extra condition that $E'(N)$ exists, which is present in Theorem \ref{th:unique}, is not needed here.}
\end{remarks*}

In order to prove Theorem \ref{th:nondeg}, we first have to show the nondegeneracy for the linearization around the ground state $\Qinf \in H^1(\RR^3)$ solving the nonrelativistic Hartree equation (\ref{eq:liebintro}). As mentioned before, this spectral result is of independent interest, since it proves a key assumption in \cite{FTY2002,JFGS2004,JFGS2006,Walid2007}. See also \cite{KriegerMartelRaphael2008}, where the following nondegeneracy result is needed. Hence we record this fact about (\ref{eq:liebintro}) as one of our main results.

\begin{theorem} {\bf (Nondegeneracy for $\Qinf$.)} \label{th:nondegnr}
Let $m > 0$ and $\lambda > 0$ be given. Furthermore, suppose that $\Qinf \in H^1(\RR^3)$ is the unique radial, positive solution to the nonrelativistic Hartree equation (\ref{eq:liebintro}). Then the linear operator $L_+$ given by
\begin{equation}
L_+ \xi = -\frac{1}{2m} \Delta \xi + \lambda \xi - \big ( |x|^{-1} \ast |\Qinf|^2 \big ) \xi - 2 \Qinf \big ( |x|^{-1} \ast (\Qinf \xi) \big )
\end{equation}
acting on $L^2(\RR^3)$ with domain $H^2(\RR^3)$, satisfies
\begin{equation}
\mathrm{ker} \, L_+ = \mathrm{span} \, \big \{ \partial_{x_1} \Qinf, \partial_{x_2} \Qinf, \partial_{x_3} \Qinf \big \} .
\end{equation}
\end{theorem}

\begin{remarks*}
{\em 1) The linearized operator $L$ for equation (\ref{eq:liebintro}) at $\Qinf$ is found to be
\begin{equation*}
L h = -\frac{1}{2m} \Delta h + \lambda h - \big ( |x|^{-1} \ast |\Qinf|^2 \big ) h - \Qinf \big ( |x|^{-1} \ast  (\Qinf ( h + \overline{h} )) \big ) .
\end{equation*}
It is convenient to view the operator $L$ (which is not $\CC$-linear) as acting on $({ \mathrm{Re} \, h \atop \mathrm{Im} \, h } )$, so that it can be written as
\begin{equation*}
L = \left ( \begin{array}{ll} L_+ &0 \\ 0 & L_- \end{array} \right ) .
\end{equation*}
Here $L_+$ is as in Theorem \ref{th:nondegnr} above, and $L_-$ is the (local) operator
\begin{equation*}
L_- = -\frac{1}{2m} \Delta + \lambda - \big ( |x|^{-1} \ast |\Qinf|^2 ) .
\end{equation*}
It is easy to see that $\mathrm{ker} \, L_- = \mathrm{span} \, \{ \Qinf \}$ holds. Hence, by Theorem \ref{th:nondegnr}, we obtain
\begin{equation*} \label{eq:kerLLgood}
\mathrm{ker} \, L = \mathrm{span} \, \Big \{ \left ( {\partial_{x_1} \Qinf \atop 0 } \right ), \left ( {\partial_{x_2} \Qinf \atop 0 } \right ), \left ( {\partial_{x_2} \Qinf \atop 0 } \right ), \left ( {0 \atop \Qinf } \right ) \Big \} .
\end{equation*}

2) The precise knowledge of $\mathrm{ker} \, L$ implies, by well-known arguments along the lines for NLS with local nonlinearities (given in \cite{Weinstein1985}), the following coercivity estimate: There is a constant $\delta > 0$ such that
\begin{equation*} 
\langle f, L_+ f \rangle + \langle g, L_- g \rangle \geq \delta ( \| f \|_{H^1}^2 + \| g \|_{H^1}^2 ),
\end{equation*}
when $f \perp \mathrm{span} \, \{ \Qinf, x_i \Qinf \}_{i=1}^3 $ and $g \perp \mathrm{span} \, \{ 2\Qinf + r \partial_r \Qinf, \partial_{x_i} \Qinf \}_{i=1}^3$, which means that $(f,g)$ is {\em symplectically orthogonal} to the ``soliton manifold'' generated by $\Qinf$; see, e.\,g., \cite{JFGS2004}. This coercivity estimate plays a central role in the stability analysis of solitary waves for NLS-type equations and their effective motion in an external potential; see, e.\,g., \cite{BronskiJerrard2000,FJLNonlinearity2007,JFGS2004,JFGS2006,HolmerZworski2007,Walid2007,Weinstein1985}. }
\end{remarks*}

\subsection*{Organization of the Paper}
This paper is structured as follows. In Section \ref{sec:nrlimit}, we study the nonrelativistic limit of ground states for a dimensionalized version of the variational problem (\ref{eq:var}). In Section \ref{sec:nondegnr}, we prove a nondegeneracy result for the nonrelativistic ground state $\Qinf \in H^1(\RR^3)$ in the radial setting. Then, in Section \ref{sec:localunique}, we establish a local uniqueness result around $\Qinf \in H^1(\RR^3)$ by means of an implicit-function-type argument. 

Finally, we prove Theorem \ref{th:unique}, \ref{th:nondeg} and \ref{th:nondegnr} in Sections \ref{sec:proof:th:unique} and \ref{sec:proof:th:nondeg}, respectively. The Appendix collects some auxiliary results and we also give a uniqueness proof for the ground state $\Qinf \in H^1(\RR^3)$, which differs from \cite{Lieb1977} in certain ways. 

\subsection*{Notation and Conventions}
As usual $H^s(\RR^n)$ stands for the inhomogeneous Sobolev space of order $s \in \RR$, equipped with norm $\| f \|_{H^s} = \| \langle \nabla \rangle^s f \|_{L^2}$, where $\langle \nabla \rangle$ is defined via its multiplier $\langle \xi \rangle = (1+\xi^2)^{1/2}$ in the Fourier domain. Also, we shall make use of the space of radial and real-valued functions that belong to $H^1(\RR^3)$, which we denote by
\[
\Honer = \{ f : f \in H^1(\RR^3), \; \mbox{$f$ is radial and real-valued} \} .
\]
With the usual abuse of notation we shall write both $f(x)$ and $f(r)$, with $r = |x|$, for radial functions $f$ on $\RR^n$. For any measurable function $f : \RR^n \to \CC$ that vanishes at infinity, we denote its symmetric-decreasing rearrangement by $f^* = f^*(r) \geq 0$.

Throughout this paper, we assume that the mass parameter $m >0$ in (\ref{eq:En}) is strictly positive, which is the physically relevant case. 

For the reader's orientation, we mention that our definition of $\En(\psi)$ in (\ref{eq:En}) differs from the conventions in \cite{LiebYau1987, FJLCMP2007} by an inessential factor of 2 and by the fact that we use $\sqrt{-\Delta + m^2} $ instead of $\sqrt{-\Delta + m^2} -m$. Obviously, these slight alterations in our definition of $\En(\psi)$ do not affect any results on (\ref{eq:var}) that are derived or quoted in the present paper.

Finally, we point out that the function $\Qinf \in H^1_r(\RR^3)$, which denotes the unique ground state for equation (\ref{eq:liebintro}), appears throughout the paper. However, for the sake of simple notation, we shall also denote all its rescaled copies $a \Qinf(b \cdot)$, with $a > 0$ and $b>0$, simply by $\Qinf$, whenever there is no source of confusion.

\subsection*{Acknowledgments}
It is a pleasure to thank Joachim Krieger and Maciej Zworski for helpful discussions, as well as Mathieu Lewin for pointing out to results on the nonrelativistic limit for Dirac-Fock equations. The author is also indebted to Mohammed Lemou and Pierre Rapha\"el, who found a gap in the previous version of this paper. This work was partially supported by the National Science Foundation Grant DMS-0702492.

\section{Nonrelativistic Limit} \label{sec:nrlimit}

As a preliminary step towards the proof of Theorems \ref{th:unique} and \ref{th:nondeg}, we study the nonrelativistic limit of ground states for the pseudo-relativistic Hartree energy functional. More precisely, we reinstall the speed of light $c > 0$ into $\En(\psi)$ defined in (\ref{eq:En}), which yields the $c$-depending Hartree energy functional
\begin{equation} \label{eq:Enc}
\Enc(\psi) = \int_{\RR^3} \overline{\psi} \Tmc \psi - \frac{1}{2} \int_{\RR^3} \big (|x|^{-1} \ast |\psi|^2 \big ) |\psi|^2 .
\end{equation}
An elementary calculation shows that, for any $\psi \in H^{1/2}(\RR^3)$,
\begin{equation} \label{eq:equiv}
\En(\psi) = c^{-3} \Enc( \widetilde{\psi} ), \quad \mbox{with $\psi(x) = c^{-2} \widetilde{\psi}(c^{-1} x)$.}
\end{equation}
Thus we immediately find the following equivalence.
\begin{lemma} \label{lemma:equivalence}
Let $c > 0$ and $N > 0$. Then $\widetilde{Q} \in H^{1/2}(\RR^3)$ minimizes $\Enc(\psi)$ subject to $\int |\psi|^2 = N$ if and only if $Q = c^{-2} \widetilde{Q}(c^{-1} \cdot)$ minimizes $\En(\psi)$ subject to $\int |\psi|^2 = c^{-1} N$.

In particular, we have existence of ground states for $\Enc(\psi)$ subject to $\int |\psi|^2 = N$ if and only if $0 < N < c \Ncrit$ holds, where $\Ncrit > 4/\pi$ denotes the same universal constant as in Theorem \ref{th:intro}.
\end{lemma}

We now study the behavior of ground states $\Qc$ for $\Enc(\psi)$ as $c \to \infty$ with $\int_{\RR^3} |\Qc|^2=N$ being fixed. By Lemma \ref{lemma:equivalence}, this is equivalent (after a suitable rescaling) to studying ground states for $\En(\psi)$ with $\int |\psi|^2 = N$ as $N \to 0$. However, the following analysis turns out to be more transparent when working with $c > 0$ as a parameter and sending $c$ to infinity.  Concerning the nonrelativistic limit $c \to \infty$ of ground states for $\Enc(\psi)$, we have the following result.

\begin{proposition} \label{prop:nrlimit}
Let $m >0$ and $N > 0$ be given, and suppose that $c_n \to \infty$ as $n \to \infty$. Furthermore, we assume that $\{ \Qcn \}_{n=1}^\infty$ is a sequence of symmetric-decreasing ground states such that $\int_{\RR^3} |\Qcn|^2 = N$ for all $n \geq 1$, and each $\Qcn \in H^{1/2}(\RR^3)$ minimizes $\En_{c_n}(\psi)$ subject to $\int_{\RR^3} |\psi|^2 = N$. Finally, let $\{ \mu_{c_n} \}_{n=1}^\infty$ denote the sequence of Lagrange multipliers corresponding to $\{ \Qcn \}_{n=1}^\infty$. 

Then the following holds:
\[
\Qcn \to \Qinf \quad {in} \quad H^1(\RR^3) \quad \mbox{as} \quad n \to \infty,
 \]
and
\[
-\mu_{c_n} - m c_n^2 \to -\lambda \quad \mbox{as} \quad n \to \infty,
\]
where $\Qinf \in H^1(\RR^3)$ is the unique radial, positive solution to
\begin{equation} \label{eq:lieb}
-\frac{1}{2m} \Delta \Qinf - \big ( |x|^{-1} \ast |\Qinf|^2 \big ) \Qinf = -\lambda \Qinf,
\end{equation}
such that $\int_{\RR^3} |\Qinf|^2 = N$. Here $\lambda > 0$ is determined through $\Qinf = \Qinf^* \in H^{1}(\RR^3)$, which is the unique symmetric-decreasing minimizer of the variational problem
\begin{equation} \label{eq:Enr}
E_{\mathrm{nr}}(N) = \inf \Big \{ \Enr(\psi) : \mbox{$\psi \in H^1(\RR^3)$ and $\displaystyle \int_{\RR^3} |\psi|^2 = N$} \Big \} ,
\end{equation}
where
\begin{equation} \label{eq:Enr2}
\Enr(\psi) = \frac{1}{2m} \int_{\RR^3} |\nabla \psi|^2 - \frac{1}{2} \int_{\RR^3} \big ( |x|^{-1} \ast |\psi|^2 \big ) |\psi|^2 .
\end{equation}
\end{proposition}

\begin{remarks*}
{\em 1) A similar result for the nonrelativistic limit of ground states (and excited states) solving the Dirac-Fock equations can be found in \cite{EstebanSere2001}. However, unlike the Dirac-Fock and Hartree-Fock energy functionals in atomic physics treated in \cite{EstebanSere2001}, the energy functional in (\ref{eq:Enc}) is not weakly lower semicontinuous due to its attractive potential term. Therefore, an a-priori bound on the sequence of Lagrange multipliers $\mu_{c_n}$ (away from the essential spectrum of the limiting equation) is not sufficient to conclude strong convergence. To deal with this, we also have to use the radial symmetry of the $\Qcn$ in order to prove strong convergence. 

2) The uniqueness of the symmetric-decreasing ground state for problem (\ref{eq:Enr}) was proven by Lieb in \cite{Lieb1977}. For the reader's convenience, we provide a (partly different) proof of this fact in Appendix A. }
\end{remarks*}

\subsection{Proof of Proposition \ref{prop:nrlimit}}

We begin with some auxiliary results.

\begin{lemma} \label{lemma:mubound}
Let $\{ \mu_{c_n} \}_{n=1}^\infty$ be as in Proposition \ref{prop:nrlimit}. Then there exist constants $\delta_1 > 0$ and $\delta_2 > 0$ such that
\[
m c_n^2 - \delta_1 \leq -\mu_{c_n} \leq m c_n^2 - \delta_2, \quad \mbox{for all $n \geq n_0$},
\]
where $n_0 \gg 1$ is some number.
\end{lemma}

\begin{proof}
The existence of $\delta_2  > 0$ can be deduced as follows. The Euler-Lagrange equation for $\Qcn$ reads 
\begin{equation} \label{eq:ELc}
\sqrt{-c_n^2 \Delta + m^2 c_n^4} \, \Qcn - \big (|x|^{-1} \ast |\Qcn|^2 \big ) \Qcn = -\mu_{c_n} \Qcn,
\end{equation}
which upon multiplication with $\Qcn$ and integration gives us
\begin{equation} \label{eq:muu1}
 \Encn (\Qcn) - \frac{1}{2} \int_{\RR^3} \big ( |x|^{-1} \ast | \Qcn |^2 \big ) |\Qcn|^2  = - \mu_{c_n} N. 
\end{equation}
Next, we recall the operator inequality 
\[
\sqrt{-c^2 \Delta + m^2 c^4} \leq -\frac{1}{2m} \Delta + m c^2,
\]
which directly follows in the Fourier domain and the fact that $\sqrt{1+t} \leq \frac{t}{2} + 1$ holds for all $t \geq 0$. Hence we have that $\Encn(\Qcn) \leq \Enr(\Qcn) + N m c_n^2$. Furthermore, since $\Qcn$ is a ground state for $\Encn(\psi)$, we deduce 
\[ \Encn(\Qcn) \leq E_{\mathrm{nr}}(N) + N m c_n^2 , \]
with $E_{\mathrm{nr}}(N)$ defined in (\ref{eq:Enr}), so that (\ref{eq:muu1}) gives us
\[
-\mu_{c_n} N \leq E_{\mathrm{nr}}(N) +  N m c_{n}^2 .
\]
From \cite{Lieb1977} we know that $E_{\mathrm{nr}}(N) < 0$ and thus $\delta_2 = -E_{\mathrm{nr}}(N)/N > 0$ is a legitimate choice.

To prove the existence of $\delta_1 > 0$, we observe that each $\Qcn \geq 0$ is the ground state of the ``relativistic'' Schr\"odinger operator
\[
H_{c_n} = \sqrt{ -c_n^2 \Delta + m^2 c_n^4} - \big ( |x|^{-1} \ast |\Qcn|^2 \big ) .
\]
Since all $\Qcn$ are radial functions with $\| \Qcn \|_{L^2}^2 = N$ for all $n \geq 1$, we can invoke Newton's theorem to find
\[
\int_{\RR^3} \frac{Ê|\Qcn(y)|^2}{|x-y|} \, dy \leq \frac{N}{|x|} .
\]
By the min-max principle, we infer the the lower bound
\[
-\mu_{c_n} \geq \inf \sigma ( \overline{H}_{c_n} )
\]
where 
\[
\overline{H}_{c_n} = \sqrt{-c_n^2 \Delta + m^2 c_n^4} - \frac{N}{|x|} .
\]
From \cite{Herbst1977} and reinstalling the speed of light $c > 0$ there, we recall that we have $\inf \sigma( \overline{H}_{c_n} ) > -\infty$ if and only if $N < \frac{2}{\pi} c_n$. Thus $\overline{H}_{c_n}$ is bounded below for $n \gg 1$ and, moreover, we have an explicit lower bound (see \cite{Herbst1977} again) given by
\[
\inf \sigma( \overline{H}_{c_n} ) \geq m c_n^2 \sqrt{1 - \left ( \frac{\pi N}{2 c_n} \right )^2 } .
\]
Since $\sqrt{1-x^2} \geq 1-x^2$ for $|x| \leq 1$, we conclude
\[
-\mu_{c_n} \geq m c_n^2 \Big ( 1 -  \big ( \frac{\pi N}{2 c_n} \big )^2 \Big ) = mc_n^2 - \frac{1}{4} m \pi^2 N^2, \quad \mbox{for all $n \geq n_0$},
\]
provided that $n_0 \gg 1$. By choosing $\delta_1 = \frac{1}{4} m  \pi^2 N^2 > 0$, we complete the proof of Lemma \ref{lemma:mubound}. \end{proof}

Next, we derive an a-priori bound on the sequence of ground states.

\begin{lemma} \label{lemma:K}
Let $\{ \Qcn \}_{n=1}^\infty$ be as in Proposition \ref{prop:nrlimit}. Then there exists a constant $M > 0$ such that
\[
\| \Qcn \|_{H^1} \leq M, \quad \mbox{for all $n \geq 1$}.
\]
\end{lemma}

\begin{proof}
Since $\| \Qcn \|_{L^2}^2 = N$ for all $n \geq 1$, we only have to derive a uniform bound for $\| \nabla \Qcn \|_{L^2}$ which can be done as follows. From equation (\ref{eq:ELc}) we obtain
\begin{align*}
c_n^2 \| \nabla \Qcn \|_{L^2}^2 + m^2 c_n^4 \| \Qcn \|_{L^2}^2 & = \langle \Tmcn \Qcn, \Tmcn \Qcn \rangle \\ \nonumber
& \leq \mu_{c_n}^2  \langle \Qcn, \Qcn \rangle  + 2 | \mu_{c_n}| \langle \Qcn, ( |x|^{-1} \ast |\Qcn|^2 ) \Qcn \rangle \\ \nonumber
& \quad + \langle \Qcn (|x|^{-1} \ast |\Qcn|^2), (|x|^{-1} \ast |\Qcn|^2) \Qcn \rangle . \label{eq:QcnH1bound}
\end{align*}
To bound the terms on the right side, we notice that Kato's inequality $|x|^{-1} \lesssim |\nabla|$ implies
\[
\| |x|^{-1} \ast |\Qcn|^2 \|_{L^\infty} \lesssim \langle \Qcn, |\nabla| \Qcn \rangle \lesssim \| \Qcn \|_{L^2} \| \nabla \Qcn \|_{L^2} .
\]
Using this bound, H\"older's inequality, and the bound $|\mu_{c_n}| \leq m c_n^2$ for $n \gg 1$ from Lemma \ref{lemma:mubound}, we obtain
\[
 c_n^2 \| \nabla \Qcn \|_{L^2}^2 \lesssim m c_n^2 N^{3/2} \| \nabla \Qcn \|_{L^2} + N^2 \| \nabla \Qcn \|_{L^2}^2 ,
\]
for $n \gg 1$. Since $c_n \to \infty$ and $N$ is fixed, we conclude that there exists $M > 0$ such that 
\[
\| \nabla \Qcn \|_{L^2} \leq M 
\]
for $n \gg 1$. By choosing $M > 0$ possibly larger, we extend this bound to all $n \geq 1$. \end{proof}

We now come the proof of Proposition \ref{prop:nrlimit} itself. By the a-priori bound in Lemma \ref{lemma:K}, we have (after possibly passing to a subsequence) that
\[
\mbox{$\Qcn \weakto \Qinf$ in $H^1(\RR^3)$ and $\Qcn(x) \to \Qinf(x)$ for a.\,e.~$x \in \RR^3$ as $n \to \infty$},
\]
for some $\Qinf \in H^1(\RR^3)$. By radiality and strict positivity of all the $\Qcn$, it follows that $\Qinf(|x|) \geq 0$ is a radial and nonnegative function. Furthermore, since $\{ \Qcn \}_{n=1}^\infty$ forms a sequence of radial functions on $\RR^3$ with a uniform $H^1$-bound, a classical result (see  \cite{Strauss1977}) yields that
\begin{equation} \label{eq:strauss}
\mbox{$\Qcn \to \Qinf$ in $L^p(\RR^3)$ as $n \to \infty$ for any $2 < p < 6$.}
\end{equation}
By Lemma \ref{lemma:mubound}, we have that $\{ -\mu_{c_n}- m c_n^2 \}_{n=1}^\infty$ is a bounded sequence, which is also uniformly bounded away from 0. Hence extracting a suitable subsequence yields
\begin{equation} \label{eq:mulimit}
\lim_{n \to \infty} (-\mu_{c_n} - m c_n^2) = -\lambda < 0,
\end{equation}
for some $\lambda > 0$. 

Using that $\Qcn \weakto \Qinf$ in $H^1$ and the strong convergence (\ref{eq:strauss}), we can pass to the limit in equation (\ref{eq:ELc}) and find that the radial, nonnegative function $\Qinf \in H^1(\RR^3)$ satisfies
\begin{equation} \label{eq:bsc2}
-\frac{1}{2m} \Delta \Qinf - \big ( |x|^{-1} \ast |\Qinf|^2  \big ) \Qinf = - \lambda \Qinf \quad \mbox{in $H^{-1}(\RR^3)$} .
\end{equation}
When taking this limit, we use the fact that 
\[
\lim_{n \rightarrow \infty} \big \langle f, \big [ \sqrt{-c_n^2 \Delta + m^2 c_n^4} - m c_n^2+ \frac{1}{2m} \Delta \big ] \Qcn \big \rangle = 0  \quad \mbox{for all $f \in H^1(\RR^3)$},
\]
which is easy to verify for test functions $f \in C_0^\infty(\RR^3)$ by taking the Fourier transform and using that
\[
 \sqrt{c_n^2 \xi^2 + m^2 c_n^4} - m c_n^2 - \frac{\xi^2}{2m} \to 0 \quad  \mbox{for every $\xi \in \RR^3$ as $c_n \to \infty$}.
\]
The claim above extends to all $f \in H^1(\RR^3)$ by a simple density argument.

Next we prove that in fact $\int |\Qinf|^2 = N$ holds, which a-posteriori would show that $\Qcn \to \Qinf$ strongly in $L^2(\RR^3)$. To prove this claim, we note that equation (\ref{eq:ELc}) and its limit (\ref{eq:bsc2}) give us
\begin{equation} \label{eq:liminf1}
(-\mu_{c_n} - m c_n^2) N = \frac{1}{2m} \int_{\RR^3} |\nabla \Qcn|^2 - \int_{\RR^3} \big ( |x|^{-1} \ast |\Qcn|^2 \big ) |\Qcn|^2 + r_n,
\end{equation}
with $r_n \rightarrow 0$ as $n \to \infty$. Note that the right-hand side is not weakly lower semicontinuous (with respect to weak $H^1$-convergence) unlike the case of atomic Hartree and Hartree-Fock energy functionals. To deal with the non weakly lower semicontinuous part given by the potential energy term, we use (\ref{eq:strauss}) again and the Hardy-Littlewood-Sobolev inequality. Then, by the weak lower semicontinuity of the kinetic energy term in (\ref{eq:liminf1}), we deduce from (\ref{eq:liminf1}) and equation (\ref{eq:bsc2}) that
\[
-\lambda N \geq  \frac{1}{2m} \int_{\RR^3} |\nabla \Qinf|^2 - \int_{\RR^3} \big ( |x|^{-1} \ast |\Qinf|^2 \big ) |\Qinf|^2 = -\lambda \int_{\RR^3} |\Qinf|^2 .
\]
Because of $\lambda > 0$, we see that $\int |\Qinf|^2  \geq N$ must hold. On the other hand, we have $N \geq \int |\Qinf|^2$ by the weak $L^2$-convergence. Thus we have $\int |\Qinf|^2 = N$ and, consequently,
\begin{equation}
\mbox{$\Qcn \to \Qinf$ in $L^2(\RR^3)$ as $n \to \infty$.}
\end{equation}
By Lemma \ref{lemma:lieb} and a simple scaling argument, we see that $\Qinf$ is the unique radial, nonnegative solution to equation (\ref{eq:bsc2}) with $\int |\Qinf|^2 = N$. Here $\lambda > 0$ is determined through $\Qinf$, and $\Qinf$ is in fact strictly positive.

It remains to show that
\begin{equation} \label{eq:H1conv}
\mbox{$\Qcn \to \Qinf$ in $H^1(\RR^3)$ as $n \to \infty$} .
\end{equation}
To see this, we verify that $\{ \Qcn \}_{n=1}^\infty$ with $\int |\Qcn|^2 = N$ furnishes a minimizing sequence for the nonrelativistic Hartree energy $\Enr(\psi)$ subject to $\int |\psi|^2 = N$, \ie, for problem (\ref{eq:Enr}). Indeed, using (\ref{eq:liminf1}) and (\ref{eq:mulimit}) as well as the strong convergence (\ref{eq:strauss}) to pass to the limit in the potential energy, we deduce that
$$
\Enr(\Qcn) \to -\lambda N + \frac{1}{2} \int_{\RR^3} \big ( |x|^{-1} \ast |Q_\infty|^2 \big ) |\Qinf|^2 \quad \mbox{as} \quad n \to \infty.
$$
However, this limit for $\Enr(\Qcn)$ is equal to $\Enr(Q_\infty)$, as can be seen by multiplying equation (\ref{eq:bsc2}) with $\Qinf$ and integrating. Hence $\{ \Qcn \}_{n=1}^\infty$ is a minimizing sequence for problem (\ref{eq:Enr}). Next, we notice  that standard concentration-compactness methods yield relative compactness in $H^1(\RR^3)$ for any radial minimizing sequence for problem (\ref{eq:Enr}), which has a unique radial, nonnegative minimizer $\Qinf$. Therefore (after possibly passing to another subsequence) we deduce that (\ref{eq:H1conv}) holds.

To conclude the proof of Proposition \ref{prop:nrlimit}, we note that we have convergence along every subsequence because of the uniqueness of the limit point $\Qinf \in H^1(\RR^3)$. \hfill $\blacksquare$

\section{Radial Nondegeneracy of Nonrelativistic Ground States} \label{sec:nondegnr}

We consider the linear operator 
\begin{equation} \label{eq:Lnr}
L_{+} \xi = - \frac{1}{2m} \Delta \xi + \lambda \xi - \big ( |x|^{-1} \ast |\Qinf|^2 \big ) \xi - 2 \Qinf \big ( |x|^{-1} \ast (\Qinf \xi ) \big ) ,
\end{equation}
where $\Qinf \in H^1(\RR^3)$ is the radial, positive solution taken from Proposition \ref{prop:nrlimit}. By standard arguments, it follows that $L_+$ is a self-adjoint operator acting on $L^2(\RR^3)$ with domain $H^2(\RR^3)$. In this section, we study the restriction of $L_+$ acting on $L^2_{\mathrm{rad}}(\RR^3)$ (i.\,e., the radial $L^2$-functions on $\RR^3$). 

As a main result, we prove the so-called nondegeneracy of $L_+$ on $L^2_{\mathrm{rad}}(\RR^3)$; that is, the triviality of its kernel.

\begin{proposition} \label{prop:nondegnr}
For the linear operator $L_+$ be given by (\ref{eq:Lnr}), we have
\[
\mathrm{ker} \, L_+ = \{ 0 \} \quad \mbox{when $L_+$ is restricted to $L^2_{\mathrm{rad}}(\RR^3)$}.
\]
\end{proposition}

\begin{remark*}
{\em 1) As shown in Section \ref{sec:proof:th:nondeg} below, we will see that the triviality of the kernel of $L_+$ on $L^2_{\mathrm{rad}}(\RR^3)$ implies 
\begin{equation}
\mathrm{ker} \, L_+ = \mathrm{span} \, \big \{ \partial_{x_1} \Qinf, \partial_{x_2} \Qinf, \partial_{x_3} \Qinf \big \} .
\end{equation}
For linearized operators arising from ground states for NLS with {\em local nonlinearities,} this fact is well-known; see \cite{CGNT2007,Weinstein1985}. However, the proof given there {\em cannot be adapted} to $L_+$ given by (\ref{eq:Lnr}) due to its nonlocal component. We refer to Section \ref{sec:proof:th:nondeg} for further details. 

2) Numerical evidence, which indicates that $0$ is not an eigenvalue of $L_+$ when restricted to radial functions, can be found in \cite{HMT2003}. }
\end{remark*}

\subsection{Proof of Proposition \ref{prop:nondegnr}}

Suppose that $\Qinf \in H^1(\RR^3)$ is the unique radial, positive solution to equation (\ref{eq:lieb}) with $\int |\Qinf|^2 = N$ for some $N >0$ given. In what follows, it will be convenient and without loss of generality to assume that $\Qinf$ satisfies
\begin{equation} \label{eq:Qinfnorm}
-\Delta \Qinf - \big ( |x|^{-1} \ast |\Qinf|^2 \big ) \Qinf = -\Qinf,
\end{equation}
which amounts to rescaling $\Qinf(x) \mapsto a \Qinf(bx)$ with suitable $a > 0$ and $b >0$. Likewise, the linear operator $L_+$ then reads
\begin{equation}
L_+ \xi = -\Delta \xi + \xi - \big (|x|^{-1} \ast |\Qinf|^2 \big ) \xi - 2 \Qinf \big ( |x|^{-1} \ast (\Qinf \xi) \big ) . 
\end{equation}

Recall that we restrict ourselves to radial $\xi \in L^2_{\mathrm{rad}}(\RR^3)$. Therefore, we can rewrite the nonlocal term in $L_+$ by invoking Newton's theorem in $\RR^3$ (see \cite[Theorem 9.7]{LiebLoss2001}): For any radial function $\rho=\rho(|x|)$ such that $ \rho \in L^1(\RR^3, (1+|x|)^{-1} dx)$, we have
\begin{equation}
-(|x|^{-1} \ast \rho)(r) = \int_0^r K(r,s) \rho(s) \, ds - \int_{\RR^3} \frac{\rho(|x|)}{|x|}  ,
\end{equation}
for $r= |x| \geq 0$, where $K(r,s)$ is given by
\begin{equation} \label{eq:K}
K(r,s) = 4 \pi s \big ( 1 - \frac{s}{r} \big ) \geq 0, \quad \mbox{for $r \geq s$} .
\end{equation}
Since the ground state $\Qinf$ is exponentially decaying, we can apply Newton's theorem to $\rho = \Qinf \xi$ for any $\xi \in L^2_{\mathrm{rad}}(\RR^3)$ and obtain the following result.

\begin{lemma} \label{lemma:LL}
For any $\xi \in L^2_{\mathrm{rad}}(\RR^3)$, we have
\begin{equation}
L_+ \xi = \mathcal{L}_+ \xi - 2 \Qinf \Big ( \int_{\RR^3} \frac{\Qinf \xi}{|x|}  \Big ) ,
\end{equation}
where $\mathcal{L}_+$ is given by
\begin{equation}
\mathcal{L}_+ \xi = -\Delta \xi + \xi - \big ( |x|^{-1} \ast |\Qinf|^2 \big ) \xi + W \xi  ,
\end{equation}
with
\begin{equation}
(W \xi)(r) = 2 \Qinf(r) \int_0^r K(r,s) \Qinf(s) \xi(s) \, ds .
\end{equation}
\end{lemma}

The following auxiliary result shows exponential growth of solutions $v$ to the linear equation $\mathcal{L}_+ v = 0$.

\begin{lemma} \label{lemma:kerL}
Suppose the radial function $v=v(r)$ solves $\mathcal{L}_+ v = 0$ with $v(0) \neq 0$ and $v'(0) = 0$. Then the function $v(r)$ has no sign change and $v(r)$ grows exponentially as $r \to \infty$. More precisely, for any $0 < \delta < 1$, there exist constants $C > 0$ and $R > 0$ such that
\[
|v(r)| \geq C e^{+\delta r}, \quad \mbox{for all $r \geq R$}.
\]
In particular, we have that $v \not \in L^2_{\mathrm{rad}}(\RR^3)$.
\end{lemma}

\begin{proof}
Since $\mathcal{L}_+ v = 0$ is a linear equation, we can assume without loss of generality that $v(0) > 0$; and moreover it is convenient to assume that $v(0) > \Qinf(0)$ holds. Next, we write $\mathcal{L}_+ v = 0$ as 
\begin{equation} \label{eq:vv}
v''(r) + \frac{2}{r} v'(r) = V(r) v(r)  + W(r),
\end{equation}
with
\begin{equation}
V(r) =  1 - (|x|^{-1} \ast |\Qinf|^2)(r),
\end{equation}
and
\begin{equation}
W(r) = 2 \Qinf(r) \int_0^r K(r,s) \Qinf(s) v(s) \, ds .
\end{equation}
Note that $\Qinf(r)$ satisfies equation (\ref{eq:vv}) with $W(r)$ being removed, i.\,e.,
\begin{equation} \label{eq:QQ}
\Qinf''(r) + \frac{2}{r} \Qinf'(r) = V(r) \Qinf(r) .
\end{equation} 

We now compare $v(r)$ and $\Qinf(r)$ as follows. An elementary calculation, using equations (\ref{eq:vv}) and (\ref{eq:QQ}), leads to the ``Wronskian-type'' identity
\begin{equation}
\big ( r^2 ( \Qinf v' - \Qinf' v) \big )' = r^2 \Qinf W,
\end{equation}
which, by integration, gives us
\begin{equation} \label{eq:precomp}
r^2 ( \Qinf v' - \Qinf' v)(r) = \int_0^r s^2 \Qinf(s) W(s) \, ds. 
\end{equation}
Hence, while keeping in mind that $\Qinf(r) > 0$, we find
\begin{equation} \label{eq:comp}
r^2 \Big ( \frac{v(r)}{\Qinf(r)} \Big )' = \frac{1}{\Qinf(r)^2} \int_0^r s^2 \Qinf(s) W(s) \, ds.
\end{equation}
From this identity we now claim that 
\begin{equation} \label{eq:comp2}
v(r) > \Qinf(r), \quad \mbox{for all $r \geq 0$}.
\end{equation}
To see this, recall that $v(0) > \Qinf(0)$ and, by continuity, we have that $v(r) > Q(r)$ for $r >0$ sufficiently small. Suppose now, on the contrary to (\ref{eq:comp2}), that there is a first intersection at some positive $r=r_*$, say, so that $v(r_*) = \Qinf(r_*)$. It is easy to see that the left-hand side of (\ref{eq:comp}) (or equivalently (\ref{eq:precomp})) has to be $\leq 0$ at $r=r_*$. On the other hand, since $v(r) > \Qinf(r) > 0$ on $[0,r_*)$, we conclude that the integral on right-hand side of (\ref{eq:comp}) at $r=r_*$ must be strictly positive. This contradiction shows that (\ref{eq:comp2}) must hold. In particular, the function $v(r)$ never changes its sign.

Next, we insert the estimate (\ref{eq:comp2}) back into (\ref{eq:comp}), which yields
\begin{equation} \label{eq:comp3}
r^2 \Big ( \frac{v(r)}{Q(r)} \Big )'(r) \geq \frac{2}{\Qinf(r)^2} \int_0^r s^2 \Qinf(s)^2 \int_0^s K(s,t) \Qinf(t)^2 \, dt \, ds .
\end{equation}
We notice that $\Qinf(r) > 0$ is the unique ground state for the Schr\"odinger operator
\begin{equation}
H = -\Delta + \widetilde{V}, \quad \mbox{with $\widetilde{V} =- |x|^{-1} \ast |\Qinf|^2.$}
\end{equation}
Since $H \Qinf = -\Qinf$ and $\widetilde{V}$ is a continuous function with $\widetilde{V} \to 0$ as $|x| \to \infty$, standard arguments show that, for any $\varepsilon > 0$, there exists a constant $A_\varepsilon > 0$ such that
\begin{equation} \label{ineq:Qupper}
\Qinf(r) \leq A_\varepsilon e^{-(1-\varepsilon) r}, \quad \mbox{for all $r \geq 0$}.
\end{equation}
Furthermore, since $\Qinf(r) > 0$ is the ground state of $H$, we can obtain the following lower bound: For any $\varepsilon > 0$, there exists a constant $B_\varepsilon > 0$ such that
\begin{equation} \label{ineq:Qlower}
\Qinf(r) \geq B_{\varepsilon} e^{-(1+\varepsilon) r}, \quad \mbox{for all $r \geq 0$}.
\end{equation}
For this classical result on ground states for Schr\"odinger operators. See, e.\,g., \cite[Theorem 3.2]{CarmonaSimon1981} where a probabilistic proof is given.

Now let $0 < \varepsilon < 1$ be given. Inserting the bounds (\ref{ineq:Qupper}) and (\ref{ineq:Qlower}) into equation (\ref{eq:comp3}), we obtain
\begin{equation}
r^2 \Big ( \frac{v(r)}{Q(r)} \Big )'(r) \geq C e^{(2-2 \varepsilon) r} \int_0^r  s^2 e^{-(2+2\varepsilon) s} \int_0^s K(s,t) e^{-(2+2\varepsilon) t} \, dt \, ds ,
\end{equation}
with some constant $C = C_\varepsilon > 0$ (we drop its dependence on $\varepsilon$ henceforth). Since the double integral on the right-hand side converges as $r \to \infty$ to some finite positve value, there exists some $a > 0$ such that 
\begin{equation}
r^2 \Big ( \frac{v(r)}{Q(r)} \Big )'(r) \geq C e^{(2- 2 \varepsilon) r} , \quad \mbox{for all $r \geq a$},
\end{equation}
with some constant $C> 0$. Integrating this lower bound and using (\ref{ineq:Qlower}) again, we find that
\begin{equation} 
v(r) \geq C \frac{e^{(1 - 3 \varepsilon) r}}{r^2},  \quad \mbox{for all $r \geq R$},
\end{equation}
with some constants $C> 0$ and $R \gg 1$. Thus, for any $0 < \delta  < 1$, we arrive at the claim of Lemma \ref{lemma:kerL} by taking $0 < \varepsilon < \frac{1}{3}(1-\delta)$ and choosing $C > 0$ appropriately. \end{proof}

With the help of Lemma \ref{lemma:kerL} we are now able to prove the triviality of the kernel of $L_+$ in the radial sector.

\begin{lemma} \label{lemma:Ltrivial}
For $L_+$ be given by (\ref{eq:Lnr}), we have that $L_+ \xi =0$ with $\xi \in L^2_{\mathrm{rad}}(\RR^3)$ implies that $\xi \equiv 0$.
\end{lemma}

\begin{proof}
Suppose there exists $\xi \in L^2_{\mathrm{rad}}(\RR^3)$ with $\xi \not \equiv 0$ such that $L_+ \xi = 0$. Then, by Lemma \ref{lemma:LL}, the function $\xi$ solves the inhomogeneous problem
\begin{equation} \label{eq:Linhom}
\mathcal{L}_+ \xi = 2 \sigma \Qinf, \quad \mbox{with $\displaystyle \sigma =  \int_{\RR^3} \frac{\Qinf \xi}{|x|}$.}
\end{equation}
Therefore,
\begin{equation} \label{eq:super}
\xi = v + w,
\end{equation}
where $w$ is any particular solution to (\ref{eq:Linhom}) and $v$ is some function such $ \mathcal{L}_+ v= 0$. As shown below, it suffices to restrict ourselves to smooth $v$ and $w$.

We shall now construct a smooth $w \in L^2_{\mathrm{rad}}(\RR^3)$ as follows. We define the smooth radial function
\begin{equation}
R = 2 \Qinf + r \partial_r \Qinf \in L^2_{\mathrm{rad}}(\RR^3),
\end{equation}
where a calculation shows that
\begin{equation}
L_+ R = -2 \Qinf .
\end{equation}
Furthermore, by applying Lemma \ref{lemma:LL} to $R$, we find
\begin{equation}
\mathcal{L}_+ R = 2(\tau-1) \Qinf, \quad \mbox{with $\displaystyle \tau = \int_{\RR^3} \frac{\Qinf R}{|x|}$.}
\end{equation}
Note that $\tau \neq 1$ must hold, for otherwise Lemma \ref{lemma:kerL} with $v = R$ (and $v(0) = R(0) =Q(0) > 0$ and $v'(0) = R'(0) = 0$) would yield that $R \not \in L^2_{\mathrm{rad}}(\RR^3)$, which is a contradiction. Thus we have found a smooth particular solution to (\ref{eq:Linhom}) given by
\begin{equation}
w = \frac{\sigma}{\tau-1} R \in L^2_{\mathrm{rad}}(\RR^3).
\end{equation}
Further, we notice that $\xi \in L^2_{\mathrm{rad}}(\RR^3)$ with $L_+ \xi = 0$ is smooth by bootstrapping this equation. Therefore, by equation (\ref{eq:super}), we conclude that $v$ has to be smooth as well. Suppose that $v \equiv 0$. Then we have $\xi = w$ and $\sigma \neq 0$
(since otherwise $w=0\neq \xi$). This, however, contradicts that $L_+ \xi = 0$ and $L_+ w = -2 \frac{\sigma}{\tau -1 } \Qinf \neq 0$. 

Thus we see that $v \not \equiv 0$ in (\ref{eq:super}), where $v'(0) = 0$ by smoothness of $v$. Suppose now that $v(0) \neq 0$. Then Lemma \ref{lemma:kerL} yields that $v \not \in L^2_{\mathrm{rad}}(\RR^3)$, which contradicts (\ref{eq:super}) together with the fact that $\xi$ and $w$ both belong to $L^2_{\mathrm{rad}}(\RR^3)$. Finally, suppose that $v(0) = 0$ holds. Then $v$ solves the equation $\mathcal{L}_+ v = 0$ with initial data $v(0) = 0$ and $v'(0)=0$. However, by a standard fixed point argument, we see that the linear integro-differential equation $\mathcal{L}_+ v = 0$ with given initial data $v(0) \in \RR$ and $v'(0)=0$ has a unique solution. Therefore $v(0)=0$ and $v'(0)=0$ implies that $v \equiv 0$. Again, we arrive at a contradiction as above. \end{proof}

Clearly, Lemma \ref{lemma:Ltrivial} completes the proof of Proposition \ref{prop:nondegnr}. \hfill $\blacksquare$

\section{Local Uniqueness around $\Qinf$} \label{sec:localunique}

Recall that $\Honer$ denotes space of radial and real-valued functions that belong to $H^1(\RR^3)$. By using Proposition \ref{prop:nondegnr}, we can now prove the following local uniqueness result for a small neighborhood around $\Qinf$ in $\Honer$.

\begin{proposition} \label{prop:localunique}
Let $m > 0$ and $N > 0$ be given. Furthermore, suppose that $\Qinf \in \Honer$ is the unique radial, positive solution to 
\begin{equation} \label{eq:lieb2}
-\frac{1}{2m} \Delta \Qinf - \big (|x|^{-1} \ast |\Qinf|^2 \big ) \Qinf = - \lambda \Qinf,
\end{equation}
with $\int |\Qinf|^2 = N$, where $\lambda > 0$ is determined through $\Qinf$. Then there exist constants $c_0 \gg 1$, $\varepsilon > 0$, and $\delta > 0$ such that the following holds. For any $(c,\mu)$ with
\[
c \geq c_0, \quad  - \lambda -\varepsilon \leq -\mu - m c^2 \leq - \lambda + \varepsilon,
\]
the equation
\begin{equation} \label{eq:lieb3}
\sqrt{-c^2 \Delta + m^2 c^4} \, Q - \big (|x|^{-1} \ast |Q|^2 \big ) Q = -\mu Q
\end{equation}
has a unique solution $Q \in \Honer$, provided that $\| Q - \Qinf \|_{H^1} \leq \delta$.
\end{proposition}

\subsection{Proof of Proposition \ref{prop:localunique}}

For $\beta \geq 0$ and $z > 0$, we define the map
\begin{equation}
G(u, \beta, z ) = u + \Res(\beta, z) g(u) , 
\end{equation}
where we set
\begin{equation}
g(u) = - \big (|x|^{-1} \ast |u|^2 ) u,
\end{equation}
and, for $\beta \geq 0$ and $z > 0$, we define the family of resolvents
\begin{equation}
\Res(\beta, z) = \left \{ \begin{array}{ll} \big (-\frac{1}{2m} \Delta + z \big )^{-1} & \mbox{if $\beta = 0,$} \\ \big ( \sqrt{-\beta^{-2} \Delta + m^2 \beta^{-4}} - m \beta^{-2} + z \big )^{-1} & \mbox{if $\beta > 0.$ }
\end{array} \right .
\end{equation}
By an elementary calculation, we verify the following equivalences:
\begin{equation}
\mbox{$Q \in \Honer$ solves (\ref{eq:lieb2}) if and only if $G(Q, 0, \lambda) = 0$},
\end{equation}
and
\begin{equation} \label{eq:equiv}
\mbox{$Q \in \Honer$ solves (\ref{eq:lieb3}) if and only if $G(Q, c^{-1}, \mu + mc^2) = 0$} .
\end{equation}

To prove Proposition \ref{prop:localunique}, we now construct an implicit function-type argument for the map
\begin{equation} \label{eq:GMap}
G : \Honer \times [0, \beta_0] \times [\lambda - \varepsilon, \lambda + \varepsilon] \to \Honer,
\end{equation}
where $\beta_0 > 0$ and $\varepsilon > 0$ are small constants. To see that indeed $G(u, \beta, z) \in \Honer$  for $u \in \Honer$, we notice that $\Res(\beta,z) : \Honer \to \Honer$, as can be seen by using the Fourier transform. That $g(u)$ maps $\Honer$ into itself follows readily from the Hardy-Littlewood-Sobolev inequality and Sobolev embeddings.  Hence (\ref{eq:GMap}) is indeed well-defined.

Next, we show that the derivative
\begin{equation}
\partial_u G(u,\beta,z) = 1 + \Res(\beta,z) \partial_u g(u) : \Honer \to \Honer
\end{equation}
depends continuously on $(u,\beta,z)$. Here $\partial_u g(u)$ acting on $\xi \in \Honer$ is found to be 
\begin{equation} \label{eq:dgu}
\partial_u g(u) \xi = - \big ( |x|^{-1} \ast |u|^2 \big ) \xi - 2 u \big ( |x|^{-1} \ast (u \xi) \big ) .
\end{equation}
By using the Hardy-Littlewood-Sobolev inequality and Sobolev embeddings, we obtain that
\begin{equation}
\| ( \partial_u g(u_1) - \partial_u g(u_2)) \xi \|_{H^1} \lesssim ( \| u_1 \|_{H^1} + \| u_2 \|_{H^1} ) \| u_1 - u_2 \|_{H^1}  \| \xi \|_{H^1} ;
\end{equation}
see, e.\,g., \cite{LenzmannWP2007} for similar estimates proving Lipschitz continuity of $g(u)$. Using this estimate, we find for $u_1, u_2, \xi \in \Honer$, $\beta_1, \beta_2 \in [0,\beta_0]$, and $z_1, z_2 > 0$, 
\begin{align}
& \| ( \partial_u G(u_1, \beta_1, z_1) -\partial_u G(u_2,\beta_2, z_2) ) \xi \|_{H^1} \label{ineq:conv} \\ & \leq  \| ( \Res(\beta_1, z_1) - \Res(\beta_2, z_2)) \partial_u g(u_1) \xi \|_{H^1}  + \| \Res(\beta_2,z_2) ( \partial_u g(u_1) - \partial_u g(u_2)) \xi \|_{H^1} \nonumber \\
& \lesssim \| \Res(\beta_1, z_1) - \Res(\beta_2, z_2) \|_{L^2 \to L^2} \| u_1 \|_{H^1}^2 \| \xi \|_{H^1} \nonumber \\ & \quad + \| \Res(\beta_2, z_2) \|_{L^2 \to L^2} ( \| u_1 \|_{H^1} + \| u_2 \|_{H^1} ) \| u_1 - u_2 \|_{H^1} \| \xi \|_{H^1} ,  \nonumber
\end{align}
where we also use the fact that $\| \Res(\beta,z) \|_{H^s \to H^s} = \| \Res(\beta,z) \|_{L^2 \to L^2}$ for any $s \in \RR$, since $\Res(\beta,z)$ commutes with $\langle \nabla \rangle$. Moreover, by using the Fourier transform, one verifies 
\begin{equation} \label{eq:Rescon}
\| \Res(\beta_1, z_1) - \Res(\beta_2,z_2) \|_{L^2 \to L^2} \to 0 \quad \mbox{as} \quad (\beta_1, z_1) \to (\beta_2,z_2) , 
\end{equation}
for any $\beta_1, \beta_2 \geq 0$ and $z_1, z_2 > 0$. (For later use, we record that (\ref{eq:Rescon}) also holds for complex $z_1, z_2 \in \CC \setminus [0,\infty)$.) Going back to (\ref{ineq:conv}), we thus find
\begin{equation*}
\| \partial_u G(u_1,\beta_1, z_1) - \partial_u G(u_2, \beta_2, z_2) \|_{H^1 \to H^1} \to 0 
\end{equation*}
as $\| u_1 - u_2 \|_{H^1} \to 0$ and $(\beta_1, z_1) \to (\beta_2, z_2)$. Hence $\partial_u G(u,\beta, z)$ depends continuously on $(u,\beta,z)$.

By Proposition \ref{prop:nondegnr} and its following remark, we have that the radial restriction of the linearized operator $L_+$ around $\Qinf$ has trivial kernel. This implies that the compact operator $(-\frac{1}{2m} \Delta + \lambda)^{-1} \partial_u g(\Qinf)$ does not have $-1$ in its spectrum. Hence the inverse operator
\begin{equation}
\big ( \partial_u G(\Qinf,0,\lambda) \big )^{-1} : \Honer \to \Honer
\end{equation}
exists. By the continuity of $\partial_u G(u,\beta,z)$ shown above, an appropriate version of an implicit function theorem (see, e.\,g., \cite{Chang2005}) implies that, for $\beta_0 > 0$ and $\varepsilon > 0$ sufficiently small, there exists a unique solution $Q=Q(\beta,z) \in \Honer$ such that
\begin{equation}
 G(Q(\beta,z), \beta, z) = 0 \quad \mbox{for $\beta \in [0,\beta_0]$ and $z \in [\lambda -\varepsilon, \lambda + \varepsilon ]$} 
\end{equation}
with
\begin{equation}
 \| Q(\beta,z) - \Qinf \|_{H^1} \leq \delta \quad \mbox{for some $\delta > 0$.}
\end{equation}
Moreover, the map $(\beta,z) \mapsto Q(\beta, z) \in \Honer$ is continuous. 

By setting $c_0 = \beta_0^{-1}$ and recalling the equivalence (\ref{eq:equiv}), we complete the proof of Proposition \ref{prop:localunique}. \hfill $\blacksquare$

\section{Proof of Theorem \ref{th:unique}} \label{sec:proof:th:unique}

First, we notice that it is sufficient to prove uniqueness of symmetric-decreasing ground states for the variational problem (\ref{eq:var}), thanks to Theorem \ref{th:intro} Part iii). Next, we make use of the rescaling correspondence formulated in Lemma \ref{lemma:equivalence}, which relates ground states for the dimensionalized and de-dimensionalized Hartree energy functionals $\Enc(\psi)$ and $\En(\psi)$ defined in (\ref{eq:Enc}) and (\ref{eq:En}), respectively. 

In what follows, we fix $\int |\Qc|^2 = 1$ and we suppose that $\Qc = \Qc^* \in \Hhalf$ is a symmetric-decreasing ground state for $\Enc(\psi)$ subject to $\int |\psi|^2 = 1$. Recall from Lemma \ref{lemma:equivalence} that $\Qc$ indeed exists for $c \geq c_0$ with $c_0$ being a sufficiently large constant. Let $\mu(Q_c)$ denote the Lagrange multiplier associated to $\Qc$ for $c \geq c_0$. We now claim that $\mu$ only depends on $c$ except for some countable set, i.\,e., we have
\begin{equation} \label{eq:mu_c}
\mu(Q_c) = \mu(c), \quad \mbox{for $c \in (c_0, \infty) \setminus \Xi$,} 
\end{equation} 
where $\Xi$ is some countable set. To prove (\ref{eq:mu_c}), we argue as follows. By Lemma \ref{lemma:equivalence}, we see that $Q = c^{-2} \Qc(c^{-1} \cdot)$ is a symmetric-decreasing ground state for $\En(\psi)$ subject to $\int |\psi|^2 = N = c^{-1}$; and moreover the Lagrange multiplier ${\mu}(Q)$ for $Q$ is found to be 
\begin{equation} \label{eq:mucorr}
{\mu}(Q) = c^{-2} \mu(Q_c). 
\end{equation}
Next, we consider the ground state energy $E(N)$ given by (\ref{eq:var}) for $0 < N < c_0^{-1}$. From \cite{LiebYau1987, FJLNonlinearity2007} we know that $E(N)$ is strictly concave. Hence $E'(N)$ exists for all $N \in (0,c_0^{-1}) \setminus \Sigma$, where is $\Sigma$ is some countable set, and we readily find that
\begin{equation} \label{eq:Enmu}
E'(N) = -\mu(Q) , \quad \mbox{for $N \in (0, c_0^{-1}) \setminus \Sigma$}. 
\end{equation}
Therefore the left-hand side of (\ref{eq:mucorr}) only depends on $N = c^{-1}$ except when $N \in \Sigma$, which proves (\ref{eq:mu_c}) with the countable set $\Xi = \{ c : \mbox{$c > c_0$ and $c^{-1} \in \Sigma$}\}$.

Suppose $\{ c_n \}_{n=1}^\infty$ is a sequence with such that $c_n \to \infty$ and values in $c_n \in (c_0, \infty) \setminus \Xi$.  Correspondingly, let $\{ \Qcn \}_{n=1}^\infty$ be a sequence of symmetric-decreasing ground states for $\Enc(\psi)$ with $\int |\Qcn|^2 = 1$ for all $n \geq 1$. By Proposition \ref{prop:nrlimit}, for any such sequence $\{ \Qcn \}$, we have that $\Qcn$ and its corresponding Lagrange multipliers $\mu_{c_n}$ satisfy the assumption of Proposition \ref{prop:localunique}, provided that $n \gg 1$. By the local uniqueness result stated in Proposition \ref{prop:localunique} and the fact $\mu_{c_n}$ only depends on $c_n$, we conclude that the symmetric-decreasing ground state $\Qc$ for $\Enc(\psi)$ subject to $\int |\psi|^2 = 1$ is unique, provided that $c \in (c_0, \infty) \setminus \Xi$ holds, where $c_0 \gg 1$ is sufficiently large and $\Xi$ is some countable set. 

Finally, by Lemma \ref{lemma:equivalence}, we deduce uniqueness of symmetric-decreasing ground states $Q$ for $\En(\psi)$ subject to $\int |\psi|^2 = N$, provided that $N \in (0,N_0) \setminus \Sigma$ holds, where $N_0 =c_0^{-1} \ll 1$ is sufficiently small and $\Sigma$ denotes some countable set. \hfill $\blacksquare$

\section{Proof of Theorems \ref{th:nondeg} and \ref{th:nondegnr}} \label{sec:proof:th:nondeg} 

We first prove Theorem \ref{th:nondegnr}. By rescaling $\Qinf(r) \mapsto a \Qinf(br)$ with suitable $a > 0$ and $b >0$, we can assume without loss of generality that $\Qinf \in \Honer$ satisfies the normalized equation
\begin{equation} \label{eq:liebnorm}
- \Delta \Qinf - \big ( |x|^{-1} \ast |\Qinf|^2 \big ) \Qinf = -\Qinf.
\end{equation} 
To complete the proof of Theorem \ref{th:nondegnr}, it suffices to prove the following result.

\begin{proposition} \label{prop:Qinfnondeg}
Let $\Qinf \in \Honer$  be the unique radial and positive solution to equation (\ref{eq:liebnorm}). Then the linearized operator $L_+$ given by
\[
L_+\xi  =  -\Delta \xi +  \xi - \big ( |x|^{-1} \ast |\Qinf|^2 \big ) \xi - 2 \Qinf \big ( |x|^{-1} \ast (\Qinf \xi) \big ),
\]
acting on $L^2(\RR^3)$ with domain $H^1(\RR^3)$, has the kernel
\[
\mathrm{ker} \, L_+ = \mathrm{span} \, \big \{ \partial_{x_1} \Qinf, \partial_{x_2} \Qinf, \partial_{x_3} \Qinf \big \} .
\]
\end{proposition}

\begin{remark*}
{\em For linearized operators $L_+$ arising from ground states $Q$ for NLS with local nonlinearities, it is a well-known fact that $\mathrm{ker} \, L_+ = \{ 0 \}$ when $L_+$ is restricted to radial functions implies that $\mathrm{ker} \, L_+$ is spanned by $\{ \partial_{x_i} Q \}_{i=1}^3$. 

The proof, however, involves some Sturm-Liouville theory which is not applicable to $L_+$ given above, due to the presence of the nonlocal term. (Also, recall that Newton's theorem is not at our disposal, since we do not restrict ourselves to radial functions anymore.) To overcome this difficulty, we have to develop Perron-Frobenius-type arguments for the action of $L_+$ with respect to decomposition into spherical harmonics. }
\end{remark*}

\subsection{Proof of Proposition \ref{prop:Qinfnondeg}}

Since $\Qinf(r)$ and $|x|^{-1}$ are radial functions, the operator $L_+$ commutes with rotations in $\RR^3$; i.\,e., we have that $(L_+ \xi (R \cdot))(x) = (L_+ \xi)(Rx)$ for all $R \in O(3)$. Therefore, we decompose any $\xi \in L^2(\RR^3)$ using spherical harmonics according to
\begin{equation} \label{eq:sph}
\xi(x) = \sum_{\ell =0 }^\infty \sum_{m=-\ell}^\ell  f_{\ell m}(r) Y_{ \ell m}(\Omega), 
\end{equation}
where $x= r \Omega$ with $r=|x|$ and $\Omega \in \mathbb{S}^2$. This gives us the direct decomposition
\begin{equation}
L^2(\RR^3) = \bigoplus \limits_{\ell = 0}^\infty \Hell \, , 
\end{equation}
so that $L_+$ acts invariantly on each 
\begin{equation}
\Hell = L^2 (\RR_+, r^2 dr  ) \otimes \mathcal{Y}_{(\ell)} .
\end{equation} 
Here $\mathcal{Y}_{(\ell)} = \mathrm{span} \, \{ Y_{\ell m} \}_{m=-\ell}^{+\ell}$ denotes the $(2 \ell + 1)$-dimensional eigenspace corresponding to the eigenvalue $\kappa_\ell = -\ell( \ell + 1)$ of the spherical Laplacian $\Delta_{\mathbb{S}^2}$ acting on $L^2(\mathbb{S}^2)$.

Let us now find an explicit formula for the action of $L_+$ on each $\Hell$. To this end, we recall the well-known the fact that
\begin{equation}
-\Delta = -\partial_{r}^2 - \frac{2}{r} \partial_r + \frac{\ell ( \ell + 1) }{r^2} \quad \mbox{on $\Hell$},
\end{equation}
as well as the multipole expansion
\begin{equation}
\frac{1}{|x - x'|} = 4 \pi \sum_{\ell = 0}^\infty \sum_{m = -\ell}^{+\ell} \frac{1}{2 \ell + 1} \frac{ r_<^\ell}{r_>^{\ell +1}} Y_{\ell m}(\Omega) Y_{\ell m}^*(\Omega') ,
\end{equation}
where $r_< = \min (|x|, |x'|)$ and $r_> = \max (|x|, |x'|)$. An elementary calculation leads to the following equivalence: We have that $L_+ \xi = 0$ if and only if
\begin{equation} \label{eq:Lequiv}
L_{+,(\ell)} f_{\ell m} = 0,\quad \mbox{for $\ell = 0,1,2, \ldots$ and $m = -\ell, \ldots, +\ell$} ,
\end{equation}
with $\xi$ given by (\ref{eq:sph}). Here the operator $L_{+,(\ell)}$ acting on $L^2(\RR_+, r^2 dr)$ is (formally) given by
\begin{equation}
(L_{+,(\ell)} f)(r) = -f''(r) - \frac{2}{r} f'(r) + \frac{\ell (\ell + 1)}{r^2} f(r) + V(r) f(r) + (W_{(\ell)} f)(r) ,
\end{equation}
with the local potential 
\begin{equation} \label{eq:Vell}
V(r) = - \big( |x|^{-1} \ast |\Qinf|^2)(r) ,
\end{equation}
and the nonlocal linear operator 
\begin{equation} \label{eq:Well}
(W_{(\ell)} f)(r) = -Ê\frac{8 \pi}{2 \ell + 1} \Qinf(r) \int_0^\infty \frac{r_<^\ell}{r_>^{\ell + 1}} \Qinf(s) f(s) \, s^2 \, ds ,
\end{equation}
where $r_< = \min(r,s)$ and $r_> = \max (r,s)$.

To prove Proposition \ref{prop:Qinfnondeg}, it suffices to assume henceforth that $\ell \geq 1$ holds, since $L_{+,(0)} f = 0$ implies that $f \equiv 0$ holds, by Proposition \ref{prop:nondegnr} above. Hence any nontrivial elements in the kernel of $L_+$ can only belong to $\Hell$ with $\ell \geq 1$. Before we proceed, we show that each $L_{+,(\ell)}$ enjoys a Perron-Frobenius property as follows.

\begin{lemma} \label{lemma:perron}
For each $\ell \geq 1$, the operator $L_{+,(\ell)}$ is essentially self-adjoint on $C^\infty_0(\RR_+) \subset L^2(\RR_+, r^2 dr)$ and bounded below. Moreover, each $L_{+,(\ell)}$ has the Perron-Frobenius property. That is, if $e_{0,(\ell)}$ denotes the lowest eigenvalue of $L_{+,(\ell)}$, then $e_{0,(\ell)}$ is simple and the corresponding eigenfunction $\phi_{0,(\ell)}(r) > 0$ is strictly positive.
\end{lemma}

\begin{remarks*} {\em 1) We have indeed the lower bound $L_{+,(\ell)} \geq 0$ for all $\ell \geq 1$. This follows from $\Hell \perp \Qinf$ for $\ell \geq 1$  and the fact that $L_+ \mid_{\Qinf^{\perp}} \geq 0$, which can be proven in the same way as for ground states for local NLS; see, e.\,g., \cite{CGNT2007, Weinstein1985}.

2) It is easy to see that $L_{+,(\ell)}$ has in fact infinitely many eigenvalues between 0 and 1. Indeed, the lower bound $\Qinf(r) \geq B_\varepsilon e^{-(1+\varepsilon)r }$ (cf.~the proof of Lemma \ref{lemma:kerL}) leads, by using Newton's theorem, to  the upper bound $V(r) \leq - \alpha r^{-1}$ with some $\alpha > 0$. Furthermore, one finds that $\langle f, W^{(\ell)} f \rangle < 0$ for $f \not \equiv 0$. Hence, we conclude
\[
L_{+,(\ell)} \leq -\partial_r^2 - \frac{2}{r} \partial_r + 1 + \frac{\ell( \ell + 1)}{r^2} - \frac{\alpha}{r} 
\]
on $L^2(\RR_+, r^2 dr)$. From the well-known spectral properties of the hydrogen atom Hamiltonian, we infer that the operator on the right side has infinitely many eigenvalues below 1, and so does $L_{+,(\ell)}$ by the min-max principle.

}
\end{remarks*}

\begin{proof}
Let us now come to the proof of Lemma \ref{lemma:perron}. Since $\Qinf(r)$ is exponentially decaying, it is straightforward to verify that $W_{(\ell)}$ is a bounded operator. Also, we have that $V \in L^\infty$ holds. Thus $L_{+,(\ell)}$ is bounded below (see also the remark following Lemma \ref{lemma:perron}). Furthermore, it is well-known that 
\begin{equation}
- \Delta_{(\ell)}= -\partial_r^2 - \frac{2}{r} \partial_r + \frac{\ell (\ell+1)}{r^2}
\end{equation}
is essentially self-adjoint on $C^\infty_0(\RR_+)$ provided that $\ell \geq 1$. In fact, this follows from \cite[Theorem X.10 and Example 4] {ReedSimon1975} which shows that $-\partial_r^2 - \frac{2}{r} \partial_r + \frac{\ell (\ell+1)}{r^2}$ is essentially self-adjoint on $C^\infty_0(\RR_+)$ if $\ell (\ell+1)/r^2 \geq 3/4r^2$. Furthermore, by the Kato-Rellich theorem and the fact that $V$ and $W_{(\ell)}$ are bounded and self-adjoint, we deduce that $L_{+,(\ell)}= -\Delta_{(\ell)} + V + W_{(\ell)}$ is essentially self-adjoint on $C^\infty_0(\RR_+)$ as well.

The Perron-Frobenius property of $L_{+,(\ell)}$ can be shown as follows. First, we consider the kinetic energy part in $L_{+,(\ell)}$, where we find that 
\begin{equation} \label{eq:DeltaPos}
\mbox{$e^{t \Delta_{(\ell)}}$ is positivity improving on $L^2(\RR_+, r^2 dr)$ for all $t > 0$.}
\end{equation}
(Recall that, by definition, this means that $e^{t \Delta_{(\ell)}} f > 0$ when $f \geq 0$ with $f \not \equiv 0$.) Indeed, an argument given in Appendix \ref{app:B} shows that the integral kernel of $e^{t \Delta_{(\ell)}}$ is strictly positive:
\begin{equation}
e^{t \Delta_{(\ell)}}(r,s) = \frac{1}{2t} \sqrt{\frac{1}{rs}} e^{- \frac{r^2 + s^2}{4t}} I_{\ell + 1/2}\Big ( \frac{rs}{2t} \Big ) > 0, \quad \mbox{for $r,s > 0$} .
\end{equation}
Here $I_k(z)$ denotes the modified Bessel function of the first kind of order $k$. For later use, we record that (\ref{eq:DeltaPos}) and the formula (by functional calculus)
\begin{equation}
(-\Delta_{(\ell)} + \mu )^{-1} = \int_0^\infty e^{-t \mu } e^{t \Delta_{(\ell)}} \, dt , \quad \mbox{for $\mu > 0$},
\end{equation}
immediately show that
\begin{equation} \label{eq:DeltaRes}
\mbox{$(-\Delta_{(\ell)} + \mu)^{-1}$ is positivity improving on $L^2(\RR_+, r^2 dr)$ for all $\mu > 0$.}
\end{equation}

Next, let $A_{(\ell)}$ denote the bounded self-adjoint operator
\begin{equation}
A_{(\ell)} = V + W_{(\ell)},
\end{equation}
where $V$ and $W_{(\ell)}$ are defined in (\ref{eq:Vell}) and (\ref{eq:Well}), respectively. Note that $A_{(\ell)}$ is nonlocal. Using that $\Qinf(r)$ is strictly positive, we readily find that
\begin{equation} \label{eq:Aell}
\mbox{$-A_{(\ell)}$ is positivity improving on $L^2(\RR_+, r^2 dr)$.}
\end{equation}
This leads to the following auxiliary result.

\begin{lemma} \label{lemma:posimprov}
For $\mu \gg 1$, the resolvent
\[
\big ( L_{+,(\ell)} + \mu \big )^{-1} = \big ( -\Delta_{(\ell)} + A_{(\ell)} + \mu \big )^{-1}
\]
is positivity improving on $L^2(\RR_+, r^2 dr)$.
\end{lemma}

\begin{proof}
For $\mu \gg 1$, we have
\begin{align*}
\frac{1}{ L_{+, (\ell)} + \mu  } =  \frac{1}{-\Delta_{(\ell)} + \mu} \frac{1}{1 + A_{(\ell)} (-\Delta_{(\ell)} + \mu)^{-1} } .
\end{align*}
Since $A_{(\ell)}$ is bounded, we conclude that $\| A_{(\ell)} (-\Delta_{(\ell)} + \mu )^{-1} \|_{L^2 \to L^2} < 1$ for $\mu \gg 1$. Thus a Neumann expansion yields
\begin{equation} \label{eq:Lres}
\frac{1}{ L_{+, (\ell)} + \mu } = \frac{1}{-\Delta_{(\ell)} + \mu} \sum_{\nu=0}^\infty \big ( - \! A_{(\ell)} (-\Delta_{(\ell)} + \mu )^{-1} \big )^\nu,
\end{equation}
provided that $\mu \gg 1$. Next, we recall from (\ref{eq:DeltaRes}) that $(-\Delta_{(\ell)} + \mu)^{-1}$ is positivity improving. By this fact and $(\ref{eq:Aell})$, we deduce from (\ref{eq:Lres}) that $(L_{+,(\ell)} + \mu)^{-1}$ must be positivity improving for  $\mu \gg 1$. This completes the proof of Lemma \ref{lemma:posimprov}. \end{proof}

We now return to the proof of Lemma \ref{lemma:perron}, which we complete as follows. Let $\ell \geq 1$ be fixed and suppose $e_{0,(\ell)} = \inf \sigma (L_{+,(\ell)})$ is the lowest eigenvalue. Furthermore, we choose $\mu \gg 1$ such that, by Lemma \ref{lemma:posimprov},
\begin{equation}
B = \big ( L_{+,(\ell)} + \mu \big )^{-1}
\end{equation}
is positivity improving on $L^2(\RR_+, r^2 dr)$. Clearly, the operator $B$ is bounded and self-adjoint, and its largest eigenvalue $\lambda_0 = \sup \sigma(B)$ is given by $\lambda_0 = ( e_{(\ell),0} + \mu)^{-1}$. Also, the corresponding eigenspaces of $L_{+,(\ell)}$ and $B$ coincide. Since $B$ is positivity improving (and hence ergodic), we can invoke \cite[Theorem XIII.43]{ReedSimon1978} to conclude that $\lambda_0$ is simple and that the corresponding eigenfunction $\phi_{(\ell),0}(r)$ is strictly positive on $\RR_+$. This proof of Lemma \ref{lemma:perron} is therefore complete.
\end{proof}

Let us now come back to the proof of Proposition \ref{prop:Qinfnondeg}, stating that $\mathrm{ker} \, L_+$ is spanned by $\{ \partial_{x_i} \Qinf \}_{i=1}^3$. By differentiating the nonlinear equation satisfied by $\Qinf$, we readily obtain that $L_+ \partial_{x_i} \Qinf = 0$ for $i=1,2,3$. Since $\partial_{x_i} \Qinf(r) = \Qinf'(r) \frac{x_i}{r} \in \mathcal{H}_{(1)}$, this show that 
\begin{equation}
L_{+,(1)} \Qinf' = 0. 
\end{equation}
Furthermore, by monotonicity of $\Qinf(r)$, we have that $\Qinf'(r) \leq 0$. Since $L_{+, (1)}$ is self-adjoint and $\Qinf'$ is an eigenfunction that does not change its sign, Lemma \ref{lemma:perron} shows that in fact $\Qinf'(r) = -\phi_{0,(1)}(r)$ holds, where $\phi_{0,(1)} > 0$ is the strictly positive ground state of $L_{+,(1)}$, with $e_{0,(1)} = 0$ being its corresponding eigenvalue. Therefore any $\xi \in \mathcal{H}_{(1)}$ such that $L_+ \xi = 0$ must be some linear combination of $\{ \partial_{x_i} \Qinf \}_{i=1}^3$.

To complete the proof of Proposition \ref{prop:Qinfnondeg}, we now claim that
\begin{equation} \label{eq:Lplus}
L_{+,(\ell)} > 0, \quad \mbox{for $\ell \geq 2$} ,
\end{equation}
which in particular shows that $L_+ \xi = 0$ with $\xi \in \Hell$ for some $\ell \geq 2$ implies that $\xi \equiv 0$. To prove (\ref{eq:Lplus}), let $\ell \geq 2$ be fixed and set
\begin{equation}
e_{0,(\ell)} = \inf \sigma(L_{+,(\ell)}). 
\end{equation}
Indeed, by the remark following Lemma \ref{lemma:perron}, we know that $e_{0,(\ell)} < 1$ is attained. (If $e_{0,(\ell)}$ was not attained, then $e_{0,(\ell)} = \inf \sigma_{\mathrm{ess}}(L_{+, (\ell)}) = 1$ and (\ref{eq:Lplus}) follows immediately.) By Lemma \ref{lemma:perron}, the eigenvalue $e_{0,(\ell)}$ is simple and its corresponding eigenfunction $\phi_{0,(\ell)}(r) > 0$ is strictly positive. Next, we notice that
\begin{equation} \label{eq:simple}
e_0 = \langle \phi_{0,(\ell)}, L_{+,(\ell)} \phi_{0,(\ell)} \rangle = \langle \phi_{0,(\ell)}, L_{+,(1)} \phi_{0,(\ell)} \rangle + K_{(\ell)},
\end{equation}
where
\begin{align*}
K_{(\ell)} & = \int_0^\infty \frac{(\ell (\ell + 1) -2)}{r^2} \phi_{0,(\ell)}(r)^2 \, r^2 \, dr \\
& \quad + 8 \pi \int_0^\infty \! \int_0^\infty \Qinf(r) \phi_{0,(\ell)}(r) \Big (\frac{1}{3} \frac{r_<}{r_>^2} - \frac{1}{2 \ell + 1} \frac{r_<^\ell}{r_>^{\ell +1}} \Big ) \Qinf(s) \phi_{0,(\ell)}(s) \, r^2 s^2 \, dr \, ds,
\end{align*}
with $r_< = \min(r,s)$ and $r_> = \max (r,s)$. Using the strict positivity of $\Qinf(r)$ and $\phi_{0,(\ell)}(r)$, we see that $K_{(\ell)} >0$ holds because of $\ell \geq 2$ and $(r_</r_>) \leq 1$. Moreover, we recall from the preceding discussion that $L_{+,(1)} \geq e_{0,(1)} = 0$. Therefore, by (\ref{eq:simple}), 
\begin{equation}
e_{0,(\ell)} \geq K_{(\ell)} > 0 , \quad \mbox{for all $\ell \geq 2$},
\end{equation}
which proves (\ref{eq:Lplus}) and completes the proof of Proposition \ref{prop:Qinfnondeg}, whence the proof of Theorem 4 follows. \hfill $\blacksquare$

\subsection{Proof of Theorem \ref{th:nondeg}}

As in the proof of Theorem \ref{th:unique} above, it is convenient to fix $N >0$ and to consider symmetric-decreasing ground state $\Qc \in \Honer$ minimizing $\Enc(\psi)$ with $\int |\Qc|^2 = N$, where we take $c >0$ sufficiently large.  In what follows, let $\mu_c$ denote the Lagrange multiplier associated to $\Qc$. (Note that it is possible that $\mu_c$ depends on $Q_c$ and not just on $c$.)

Recall from Proposition \ref{prop:nrlimit} that
\begin{equation} \label{eq:convergence}
\| \Qc - \Qinf \|_{H^1} \leq \delta_1 \quad \mbox{and} \quad  |-\mu_c - m c^2 + \lambda| \leq \delta_2 ,
\end{equation} 
where $\delta_1 \to 0$ and $\delta_2 \to 0$ as $c \to \infty$. Here $\Qinf \in \Honer$ is the unique radial positive solution to equation (\ref{eq:lieb}) with $\int |\Qinf|^2 = N$, where $\lambda >0$ is determined through $\Qinf$. By Theorem \ref{th:nondegnr}, the linear operator $L_+$ given by
\begin{equation}
L_+  \xi = - \frac{1}{2m} \Delta + \lambda - \big ( |x|^{-1} \ast |\Qinf|^2 \big ) \xi - 2 \Qinf \big (Ê|x|^{-1} \ast (\Qinf \xi) \big) 
\end{equation}
has the kernel
\begin{equation}
\mathrm{ker} \, L_+ = \mathrm{span} \, \{ \partial_{x_1} \Qinf, \partial_{x_2} \Qinf, \partial_{x_3} \Qinf \}. 
\end{equation}
Next, let $L_{+,c}$ denote the linear operators defined as
\begin{equation}
L_{+,c} \xi =  \sqrt{-c^2 \Delta + m^2 c^4} \, \xi + \mu_c \xi - \big ( |x|^{-1} \ast |\Qc|^2 \big ) \xi - 2 \Qc \big ( |x|^{-1} \ast (\Qc \xi ) \big ) .
\end{equation}
Again, upon differentiating the Euler-Lagrange equation satisfied by $\Qc$, we see that $L_{+,c} \partial_{x_i} \Qc = 0$ for $i=1,2,3$. Hence
\begin{equation} \label{eq:kerinc}
\mathrm{span} \, \{ \partial_{x_1} \Qc, \partial_{x_2} \Qc, \partial_{x_3} \Qc \big \} \subseteq \mathrm{ker} \, L_{+,c} .
\end{equation} 
By the following perturbation argument, we show that in fact equality holds for $c \gg 1$. By standard arguments, we see that $0 \in \sigma(L_+)$ is an isolated eigenvalue. Thus we can construct the Riesz projection $P_0$ onto $\mathrm{ker} \, L_+$ by
\begin{equation}
P_0 = \frac{1}{2 \pi i} \oint_{\Gamma_r} (L_+ - z)^{-1} \, dz,
\end{equation}
where the curve $\Gamma_r$ parametrizes the circle $\{ z \in \mathbb{C} : |z| = r \}$. Here $r > 0$ is chosen sufficiently small such that $0$ is the only eigenvalue of $L_+$ inside $|z| \leq r$. Next, we claim that the projection
\begin{equation} \label{eq:riesz}
P_{0,c} = \frac{1}{2 \pi i} \oint_{\Gamma_r} (L_{+,c} - z)^{-1} \, dz
\end{equation}
 exists for $c \gg 1$ and satisfies
\begin{equation} \label{eq:Pconv}
\| P_{0,c} - P_{0} \|_{L^2 \to L^2} \to 0 \quad \mbox{as} \quad c \to \infty.
\end{equation}
Indeed, by using (\ref{eq:convergence}) and similar arguments as in the proof of Proposition \ref{prop:localunique} (see, e.\,g., the resolvent estimate (\ref{eq:Rescon})), we conclude that
\begin{equation}
\| (L_{+,c} - z)^{-1} \|_{L^2 \to L^2} \leq C \| (L_+ - z)^{-1} \|_{L^2 \to L^2} ,
\end{equation}
for all $c \gg 1$ and $z \in \Gamma_r$, where $C > 0$ is some constant. Furthermore, we have
\begin{equation}
\| (L_{+,c} - z)^{-1} - (L_+ - z)^{-1} \|_{L^2 \to L^2} \to 0 \quad \mbox{as} \quad c \to \infty,
\end{equation} 
for all $z \in \Gamma_r$. This shows that $P_{0,c}$ exists for $c \gg 1$ and that (\ref{eq:Pconv}) holds. Since $\mathrm{rank} \, P_0 = 3$ and the rank of $P_{0,c}$ remains constant for $c \gg 1$, by (\ref{eq:Pconv}), we infer that $P_{0,c}$ has at most 3 eigenvalues (counted with their multiplicity) inside $|z| \leq r$, provided that $c \gg 1$. In particular, we conclude that $\dim \mathrm{ker} \, L_{+,c} \leq 3$ for $c \gg 1$. Therefore equality must hold in (\ref{eq:kerinc}) whenever $c \gg 1$. 

Thus we have found that $L_{+,c}$ has the desired kernel property if $c \gg 1$. By a rescaling argument formulated in Lemma \ref{lemma:equivalence}, we conclude the analogous statement for the linear operator $L_+$ arising from the unique symmetric-decreasing ground state $Q$ minimizing $\En(\psi)$ subject to $\int |\psi|^2 = N$ with $N \ll 1$. The proof of Theorem \ref{th:nondeg} is now complete. \hfill $\blacksquare$

\begin{appendix}

\section{Uniqueness of $\Qinf$}

Suppose that $\Qinf \in H^1(\RR^3)$ solves 
\begin{equation}
-\frac{1}{2m} \Delta \Qinf - \big ( |x|^{-1} \ast |\Qinf |^2 \big ) \Qinf = -\lambda \Qinf,
\end{equation}
with $m > 0$ and $\lambda > 0$ given. By rescaling $\Qinf(r) \mapsto a \Qinf(br)$ with suitable $a > 0$ and $b > 0$, we can consider without loss of generality solutions $\Qinf \in H^1(\RR^3)$ to the ``normalized'' equation
\begin{equation} \label{eq:liebapp}
-\Delta \Qinf - \big (|x|^{-1} \ast |\Qinf|^2 \big ) \Qinf = -\Qinf .
\end{equation}
The following result is due to \cite{Lieb1977}; see also \cite{MorozTod1999}. Here we provide a partly different proof, which is directly based on a comparison argument.

\begin{lemma} \label{lemma:lieb}
The equation (\ref{eq:liebapp}) has a unique radial, nonnegative solution $Q \in \Honer$ with $Q \not \equiv 0$. Moreover, we have that $Q(r)$ is in fact strictly positive.
\end{lemma}

\begin{proof}
Existence of a nonnegative, nontrivial solution $\Qinf \in \Honer$ of (\ref{eq:liebapp}) follows from variational arguments; see \cite{Lieb1977}.

To prove that that any nonnegative $Q \in H^1(\RR^3)$, with $Q \not \equiv 0$, solving (\ref{eq:liebapp}) is strictly positive, we can simply argue as follows. We rewrite (\ref{eq:liebapp}) as
\begin{equation} \label{eq:Qinv}
Q(x) = ((-\Delta + 1)^{-1} (V Q))(x)= \frac{1}{4 \pi} \int_{\RR^3} \frac{e^{-|x-y|}}{|x-y|} V(y) Q(y) \, dy
\end{equation}
with $V = |x|^{-1} \ast |Q|^2$. Since $V \geq 0$ and $Q \geq 0$ (with $V \not \equiv 0$ and $Q \not \equiv 0$), equation (\ref{eq:Qinv}) shows that $Q$ is strictly positive.

Let us now prove the claimed uniqueness. Suppose $Q \in \Honer$, with $Q \not \equiv 0$, is a solution to (\ref{eq:liebapp}). Using Newton's theorem, we find that $Q(r)$ solves (after a suitable rescaling $Q(r) \mapsto a^2 Q(ar)$ for some $a > 0$; see \cite{Lieb1977}) the following initial-value problem
\begin{equation} \label{eq:ivp}
\left \{ \begin{array}{ll} \displaystyle -v''(r) - \frac{2}{r} v'(r) - v(r) + \big( \int_0^r K(r,s) v(s)^2 \, ds \big ) v(r) = 0, \\[1ex]
v(0) = v_0, \quad v'(0) = 0 ,
\end{array} \right .
\end{equation}
with $v_0 = Q(0) \in \RR$. (Recall that $K(r,s) \geq 0$ is given by (\ref{eq:K}) above.) By standard fixed point arguments, we deduce that (\ref{eq:ivp}) has a unique local $C^2$-solution for given initial data $v(0) \in \RR$ and $v'(0)=0$, and $v(r)$ exists up to some maximal radius $R \in (0,\infty]$.

Suppose now that $Q \in \Honer$ and $\widetilde{Q} \in \Honer$ are two radial, nonnegative (and nontrivial) solutions to (\ref{eq:liebapp}) with $Q \not \equiv \widetilde{Q}$. From the preceding discussion we know that $Q$ and $\widetilde{Q}$ are in fact strictly positive, and (after appropriate rescaling) both satisfy (\ref{eq:ivp}) with $v_0 = Q(0) > 0$ and $v_0 = \widetilde{Q}(0) > 0$, respectively. By  uniqueness for (\ref{eq:ivp}), we conclude that $Q(0) \neq \widetilde{Q}(0)$ holds, since otherwise $Q \equiv \widetilde{Q}$. Therefore, we can henceforth assume that 
\begin{equation} \label{eq:Qinit}
\widetilde{Q}(0) > Q(0) .
\end{equation}

Next, we notice that a calculation (similar to the one in the proof of Lemma \ref{lemma:kerL}) yields the integrated ``Wronskian-type'' identity
\begin{equation} \label{eq:wronsk}
r^2 ( Q(r) \widetilde{Q}'(r) - Q'(r) \widetilde{Q}(r) )  = \int_0^r s^2 Q(s) \widetilde{Q}(s) ( \widetilde{V}(s) - V(s) ) \, ds .
\end{equation}
Here,
\begin{equation}
V(r) = \int_0^r K(r,s) Q(s)^2 \,ds \quad \mbox{and} \quad \widetilde{V}(r) = \int_0^r K(r,s) \widetilde{Q}(s)^2 \, ds.
\end{equation}
By continuity and (\ref{eq:Qinit}), we have $\widetilde{Q}(r) > Q(r)$ at least initially for $r \geq 0$. Next, we conclude, by (\ref{eq:wronsk}), that in fact
\begin{equation} \label{eq:comparison}
\widetilde{Q}(r) > Q(r),  \quad \mbox{for all $r \geq 0$}.
\end{equation}
To see this, suppose on the contrary that $\widetilde{Q}(r) > 0$ intersects $Q(r) > 0$ for the first time at $r=r_* > 0$, say. Then the left-hand side of (\ref{eq:wronsk}) is found to be nonnegative at $r=r_*$, whereas the right-hand side must be strictly positive at $r=r_*$ since $\widetilde{V}(r) > V(r)$ on $(0,r_*)$. This contradiction shows that (\ref{eq:comparison}) holds.

Finally, we show that (\ref{eq:comparison}) leads to a contradiction (along the lines of \cite{Lieb1977}) as follows. To this end, we consider the Schr\"odinger operators
\begin{equation}
H = -\Delta + V \quad \mbox{and} \quad \widetilde{H} = -\Delta + \widetilde{V} ,
\end{equation}
so that $H Q = Q$ and $\widetilde{H} \widetilde{Q} = \widetilde{Q}$. By standard theory of Schr\"odinger operators, we conclude that $Q$ and $\widetilde{Q}$ are (up to  a normalization factor) the unique positive ground states (with eigenvalue $e=1$) for $H$ and $\widetilde{H}$, respectively. Therefore,
\begin{equation} \label{eq:HR}
\langle \phi, H \phi \rangle \geq \| \phi \|_{L^2}^2 \quad \mbox{and} \quad \langle \phi, \widetilde{H} \phi \rangle \geq \| \phi \|_{L^2}^2 , \quad \mbox{for $\phi \in H^1(\RR^3)$,}
\end{equation}
where equality holds if and only if $\phi = \lambda Q$ or $\phi = \lambda \widetilde{Q}$ for some $\lambda \in \CC$, respectively. 

Going back to (\ref{eq:comparison}), we find that $\widetilde{V}(r) > V(r)$ for all $r > 0$, which leads to
\begin{align*}
\| \widetilde{Q} \|_{L^2}^2 & \leq \langle \widetilde{Q}, H \widetilde{Q} \rangle = \langle \widetilde{Q}, \widetilde{H} \widetilde{Q} \rangle - \langle \widetilde{Q}, (\widetilde{V}-V) \widetilde{Q} \rangle = \| \widetilde{Q} \|_{L^2}^2 - \delta,
\end{align*}
for some $\delta > 0$, which is a contradiction. 

Hence equation (\ref{eq:liebapp}) does not admit two different radial and nonnegative (and nontrivial) solutions $Q \in \Honer$ and $\widetilde{Q} \in \Honer$. \end{proof}

\section{Decomposition of $e^{t \Delta}$ using Spherical Harmonics} \label{app:B}

Recall the explicit formula for the heat kernel of the Laplacian $\Delta$ on $\RR^3$:
\begin{equation} \label{eq:heat}
e^{t \Delta}(x,y) = \frac{1}{(4 \pi t)^{3/2}} e^{- \frac{|x-y|^2}{4t}} = \frac{1}{(4 \pi t)^{3/2}} e^{-\frac{x^2 + y^2}{4t}} e^{\frac{ x\cdot y}{2t}} .
\end{equation}
Moreover, we have the well-known identity
\begin{equation} \label{eq:bessel}
e^{a x \cdot y} = 4 \pi \sum_{\ell=0}^\infty \sum_{m=-\ell}^{+\ell}  i_\ell(a |x| |y|) Y_{\ell m}(\Omega) Y_{\ell m}^*(\Omega')
\end{equation}
for $a > 0$, $x = |x| \Omega$ and $y = |y| \Omega'$ where $\Omega, \Omega' \in \mathbb{S}^2$. Here 
\begin{equation}
i_\ell(z) = \sqrt{ \frac{\pi}{2 z}} I_{\ell+1/2}(z)
\end{equation}
is the modified spherical Bessel function of the first kind of order $\ell$; whereas $I_k(z)$ denotes the modified Bessel function of the first kind of order $k$. 

Let $\Delta_{(\ell)}$ denote the restriction of $\Delta$ acting on $\Hell$ (i.\,e., the space of $L^2(\RR^3)$ functions whose ``angular momentum'' is $\ell \geq 0$). From (\ref{eq:heat}) and (\ref{eq:bessel}) we deduce that the integral kernel of $e^{t \Delta_{(\ell)}}$ acting on $L^2(\RR_+, r^2 dr)$ is given by
\begin{equation} \label{eq:eheatl}
e^{t \Delta_{(\ell)}}(r,s) =  \frac{1}{2t} \sqrt{\frac{1}{rs}} e^{- \frac{r^2 + s^2}{4t}} I_{\ell + 1/2}\Big ( \frac{rs}{2t} \Big ) .
\end{equation}
An explicit integral representation for $I_{k}(z)$ shows that $I_{\ell + 1/2}(z) > 0$ for all $z > 0$ and $\ell \geq 0$. 

\end{appendix}

\bibliography{Unique}
\bibliographystyle{amsalpha}

\end{document}